\documentclass[a4paper,10pt, oneside]{article}
\usepackage{amsmath,amssymb,amsthm,bm,bbm,mathrsfs,mathtools,graphicx}
\usepackage{algorithm,algorithmic}
\usepackage{authblk}
\usepackage[font=small,labelfont=md,textfont=it]{caption}
\usepackage[titletoc, title]{appendix}
\usepackage{enumitem}
\usepackage[colorlinks,linkcolor=black,citecolor=black]{hyperref}
\usepackage{etoolbox,tcolorbox}
\usepackage{longtable}
\usepackage[a4paper,paperheight=297mm,paperwidth=210mm,margin=1.25in]{geometry} 
\usepackage{diagbox}
\usepackage{booktabs,makecell,multirow}
\usepackage{cases,color}
\usepackage[capitalise]{cleveref}

\crefname{assumption}{Assumption}{Assumptions}
\crefname{lemma}{Lemma}{Lemmas}
\crefname{theorem}{Theorem}{Theorems}
\crefname{discr}{Discretization}{Discretizations}

\crefrangeformat{lemma}{Lemmas~\textup{#3(#1)#4--#5(#2)#6}}

\crefformat{equation}{\textup{#2(#1)#3}}
\crefrangeformat{equation}{\textup{#3(#1)#4--#5(#2)#6}}
\crefmultiformat{equation}{\textup{#2(#1)#3}}{ and \textup{#2(#1)#3}}
{, \textup{#2(#1)#3}}{, and \textup{#2(#1)#3}}
\crefrangemultiformat{equation}{\textup{#3(#1)#4--#5(#2)#6}}%
{ and \textup{#3(#1)#4--#5(#2)#6}}{, \textup{#3(#1)#4--#5(#2)#6}}{, and \textup{#3(#1)#4--#5(#2)#6}}

\Crefformat{equation}{#2Equation~\textup{(#1)}#3}
\Crefrangeformat{equation}{Equations~\textup{#3(#1)#4--#5(#2)#6}}
\Crefmultiformat{equation}{Equations~\textup{#2(#1)#3}}{ and \textup{#2(#1)#3}}
{, \textup{#2(#1)#3}}{, and \textup{#2(#1)#3}}
\Crefrangemultiformat{equation}{Equations~\textup{#3(#1)#4--#5(#2)#6}}%
{ and \textup{#3(#1)#4--#5(#2)#6}}{, \textup{#3(#1)#4--#5(#2)#6}}{, and \textup{#3(#1)#4--#5(#2)#6}}

\crefdefaultlabelformat{#2\textup{#1}#3}

\apptocmd{\sloppy}{\hbadness 10000\relax}{}{}

\newcommand{\dual}[1]{\langle {#1} \rangle}
\newcommand{\dualb}[1]{\big\langle {#1} \big\rangle}
\newcommand{\dualB}[1]{\Big\langle {#1} \Big\rangle}

\newcommand{\norm}[1]{\lVert {#1} \rVert}
\newcommand{\nmb}[1]{\big\lVert {#1} \big\rVert}
\newcommand{\nmB}[1]{\Big\lVert {#1} \Big\rVert}

\newcommand{\snm}[1]{\lvert {#1} \rvert}
\newcommand{\snmb}[1]{\big\lvert {#1} \big\rvert}
\newcommand{\snmB}[1]{\Big\lvert {#1} \Big\rvert}

\newcommand{\ssnm}[1]
{
	\left\vert\kern-0.25ex
	\left\vert\kern-0.25ex
	\left\vert
	{#1}
	\right\vert\kern-0.25ex
	\right\vert\kern-0.25ex
	\right\vert
}

\makeatletter
\def\spher@harm#1{%
	\vbox{\hbox{%
			\offinterlineskip
			\valign{&\hb@xt@2\p@{\hss$##$\hss}\vskip.2ex\cr#1\crcr}%
		}\vskip-.36ex}%
}
\def\gshone{\spher@harm{.}}
\def\gshtwo{\spher@harm{.&.}}
\def\gshthree{\spher@harm{.&.&.}}
\let\gsh\spher@harm
\makeatother

\newtheorem{lemma}{Lemma}[section]
\newtheorem{remark}{Remark}[section]
\newtheorem{theorem}{Theorem}[section]

\numberwithin{equation}{section}

\makeatletter\def\@captype{table}\makeatother

\begin{document}
\title{
  \Large\bf Stability and convergence of the Euler scheme for stochastic linear evolution
  equations in Banach spaces\thanks{
    This work was partially supported by National Natural
    Science Foundation of China under grant 12301525 and
    Natural Science Foundation of Sichuan Province under grant 2023NSFSC1324.
  }
}

\author{Binjie Li\thanks{libinjie@scu.edu.cn} } 
\author{Xiaoping Xie\thanks{Corresponding author: xpxie@scu.edu.cn}}
\affil{School of Mathematics, Sichuan University, Chengdu 610064, China}

\date{}
\maketitle
\begin{abstract}
  For the Euler scheme of the stochastic linear evolution
  equations, the discrete stochastic maximal $ L^p $-regularity
  estimate is established, and a sharp error estimate in the norm
  $ \norm{\cdot}_{L^p(\Omega\times(0,T);L^q(\mathcal O))} $,
  $ p,q \in [2,\infty) $, is derived via a duality argument.
\end{abstract}


\section{Introduction}

The numerical methods of stochastic partial differential equations (SPDEs) have been
extensively studied in the past decades, and by now it is still an active
research area; see, e.g., \cite{Bessaih2019,Breit2021,CarelliProhl2012,Prohl2013,Prohl2012,Diening2023,Gyongy1999,Gyongy2003,Gyongy2009,Gyongy1997,Kruse2014book,Yan2005,Zhang2017book} and the reference therein.
However, the majority of numerical analysis efforts in this domain have primarily focused on the Hilbert space
framework, while investigations within the context of Banach spaces remain relatively limited.

We summarize some related works in this field that have come to our attention.
Regarding abstract stochastic Cauchy problems, Cox and van Neerven \cite{Cox2010,Cox2013} established
pathwise Hölder convergence for both the splitting scheme and the implicit-linear Euler scheme.
Blömker and Jentzen \cite{Blomker2013} conducted a detailed analysis of Galerkin approximations
for the one-dimensional stochastic Burgers equation, focusing on the spatial $L^\infty$-space setting.
Br\'ehier et al.~\cite{Breit2019} analyzed semidiscrete splitting approximations for the stochastic Allen-Cahn equation with additive noise under general spatial $ L^q $-norms.
Recently, van Neerven and Veraar \cite{Neerven2022} developed an elegant framework for establishing pathwise uniform convergence for time discretisation schemes for a broad class of SPDEs.
Additionally, Klioba and Veraar~\cite{Klioba2024} analyzed temporal approximations of stochastic evolution equations with irregular nonlinearities within the $2$-smooth Banach spaces setting.


Despite these significant advancements, the numerical analysis of SPDEs in the broader context
of general Banach spaces remains an underdeveloped area.
This motivates us to analyze the stability and convergence of the Euler scheme for the stochastic linear
evolution equations in Banach spaces, which is one of the most popular
temporal discretization scheme in this realm.




Firstly, we establish a discrete stochastic maximal $ L^p $-regularity estimate.
Maximal $ L^p $-regularity is of fundamental importance for the deterministic
evolution equations; see, e.g., \cite{Denk2013book,Kunstmann2004,Pruss2016,Weis2001}.
In the past twenty years, the discrete maximal $ L^p $-regularity of deterministic
evolution equations has also attracted great attention; see, e.g.,
\cite{Blunck2001,Kemmochi2016, Lubich2016,Kemmochi2018,Vexler_Lp_2017,
LiB2017Math,LiB2017SIAM}.
Utilizing the techniques of $ H^\infty $-calculus, $ \mathcal R $-boundedness,
and square function estimates, van Neerven et al.~\cite{Neerven2012} have established the
following seminal stochastic maximal $ L^p $-regularity estimate:
\begin{align*}
  \Biggl[
    \mathbb E \int_{\mathbb R_{+}} 
    \biggl\lVert A^{1/2} \int_0^t S(t-s) f(s) \, \mathrm{d}W(s) \biggr\rVert_{L^q(\mathcal O)}^p
    \, \mathrm{d}t
  \Biggr]^{1/p}
  \leqslant c \Bigl[ \mathbb E \norm{f}_{L^p(\mathbb R_{+};L^q(\mathcal O;H))}^p \Bigr]^{1/p}.
\end{align*}
In this inequality, $p \in (2,\infty)$, $q \in [2,\infty)$, $\mathcal{O}$ is a bounded domain in
$ \mathbb R^d $ ($d\geqslant 2$), $ H $ is a separable Hilbert space, $ W $ represents an $ H $-cylindrical Brownian motion,
$ S(\cdot) $ is the analytic semigroup generated by a sectorial operator $ A $ on $ L^q(\mathcal O) $,
and $ f $ is an $ L^q(\mathcal O;H) $-valued stochastic process adapted to the underlying filtration.
Building on the methodology laid out in \cite{Neerven2012}, we establish the following discrete stochastic maximal $ L^p $-regularity estimate:
\begin{align*}
  \Biggl[
    \mathbb E\sum_{j=1}^\infty
    \biggl\lVert A^{1/2}\sum_{k=0}^{j-1}\int_{k\tau}^{k\tau+\tau} (I + \tau A)^{k-j} f_k \, \mathrm{d}W(t) \biggr\rVert_{L^q(\mathcal O)}^p
  \Biggr]^{1/p} \leqslant
  c \biggl[
    \mathbb E\sum_{j=0}^\infty \norm{f_j}_{L^q(\mathcal O;H)}^p
  \biggr]^{1/p},
\end{align*}
where $ \tau $ represents the time step. This result is presented rigorously in Theorem~\ref{thm:space-regu}.
For comparison, we provide the deterministic discrete maximal $L^p$-regularity estimate as follows (see \cite{Kemmochi2016}):
\begin{align*}
\Biggl[
\sum_{j=1}^\infty
\tau \biggl\lVert A\sum_{k=0}^{j-1}\int_{k\tau}^{k\tau+\tau} (I + \tau A)^{k-j} g(t) \, \mathrm{d}t \biggr\rVert_{L^q(\mathcal{O})}^p
\Biggl]^{1/p} \leqslant
c \lVert g \rVert_{L^p(\mathbb{R}+;L^q(\mathcal{O}))}, \quad \forall g \in L^p(\mathbb{R}+;L^q(\mathcal{O})).
\end{align*}
In contrast to the deterministic case, the discrete stochastic maximal $ L^p $-regularity estimate exhibits
inferior spatial regularity, a discrepancy that can be attributed to the presence of Brownian motion.
Notably, when $ p = q = 2 $ and $ A $ corresponds to the negative Laplacian on $ L^q(\mathcal{O}) $ with
homogeneous Dirichlet boundary conditions, the aforementioned discrete stochastic maximal $ L^p $-regularity
estimate is well-established and can be derived using a simple energy argument.
Furthermore, analogous estimates have been obtained in the Hilbert space setting by Kazashi~\cite{Kazashi2018}.
Although our numerical analysis assumes that $ A $ is a sectorial operator on $ L^q(\mathcal{O}) $, the findings can be extended to the case where $ A $ is the negative Stokes operator.

Secondly, we establish a sharp error estimate in the norm
$ \norm{\cdot}_{L^p(\Omega\times(0,T);L^q(\mathcal O))} $, with $ p,q \in [2,\infty) $.
Previous research, including the works of \cite{Cox2010,Cox2013,Neerven2022},
has yielded various error estimates involving general spatial $ L^q $-norms.
However, the convergence in the specific norm $ L^p(\Omega\times(0,T);L^q(\mathcal O)) $ has not been thoroughly investigated.
This type of error estimate is particularly significant for the numerical analysis of stochastic optimal control problems that involve stochastic evolution equations.
In this study, by employing a duality argument and assuming that the process $ f $ is piecewise constant in time, we derive the following sharp error estimate:
\[
\left[
\mathbb{E} \sum_{j=0}^{J-1} \int_{j\tau}^{j\tau+\tau} \norm{y(t) - Y_j}_{L^q(\mathcal O)}^p \, \mathrm{d}t
\right]^{1/p} \leqslant c \tau^{1/2}
\left[
\mathbb{E} \int_0^T \norm{f(t)}_{L^q(\mathcal O;H)}^p \, \mathrm{d}t
\right]^{1/p},
\]
where $ y $ denotes the mild solution of a stochastic linear evolution equation,
and $ (Y_j)_{j=0}^J $ represents its temporal approximation via the Euler scheme.
This result is formally stated in \cref{thm:conv}.
It is widely recognized that for the Euler scheme, achieving a convergence rate of \( 1/2 \) under general spatial
\( L^q \)-norms is typically unattainable when dealing with rough data.
However, it is noteworthy that under the additional assumption of piecewise constant temporal behavior for
the process \( f \), the convergence rate of \( 1/2 \) can still be achieved.




The remainder of this paper is structured as follows. Section \ref{sec:pre} introduces
necessary notation alongside the concepts of $\gamma$-radonifying operators, $\mathcal{R}$-boundedness,
$H^\infty$-calculus, and stochastic integral. In Section \ref{sec:stability},
we establish the discrete maximal $L^p$-regularity. Finally, Section \ref{sec:convergence} provides
a sharp error estimate.

\section{Preliminaries}
\label{sec:pre}
\medskip\noindent\textbf{Conventions}.
Throughout this paper, we will use the following conventions:
For any Banach spaces $ E_1 $ and $ E_2 $, $ \mathcal L(E_1,E_2) $ denotes the space of all bounded linear operators from $ E_1 $ to $ E_2 $, and $ \mathcal L(E_1,E_1) $ is abbreviated to $ \mathcal L(E_1) $;
The symbol $ I $ denotes the identity operator;
For each \( p \in [1,\infty] \), its conjugate exponent is denoted by \( p' \);
For any measure space $ (X,\mathcal A,\mu) $, any Banach space $ E $,
and any $ p \in [1,\infty] $, we use $ L^p(X;E) $ to denote a standard Bochner space (see \cite[Chapter 1]{HytonenWeis2016});
Let $ \mathcal O \subset \mathbb R^d \, (d\geqslant 2) $ be a bounded domain with a Lipschitz boundary;
The imaginary unit is denoted by $i$;
For any \( z \in \mathbb C \setminus \{0\} \), its argument
\( \operatorname{Arg} z \) is restricted to the interval \( (-\pi, \pi] \);
The symbol $ c $ denotes a generic positive constant, which is independent of the time step $ \tau $ but may differ in different places.
In addition, for any $ \theta \in (0,\pi) $, we define the sector
\[ 
  \Sigma_{\theta} := \{
    z \in \mathbb C \setminus \{0\} \mid
    -\theta < \operatorname{Arg} z < \theta
  \}.
\]

\medskip\noindent\textbf{$\gamma$-Radonifying operators}.
For any Banach space $ E $ and Hilbert space $ U $ with inner product $
(\cdot,\cdot)_U $, define
\[
   \mathcal S(U,E) := \text{span}\big\{
   u \otimes e \mid \, u \in U, \, e \in E
   \},
\]
where $ u \otimes e \in \mathcal L(U,E) $ is defined by
\[
  (u \otimes e)(v) := (v,u)_U e, \quad \forall v \in U.
\]
Let $ \gamma(U,E) $ denote the completion of $ \mathcal S(U,E) $
with respect to the norm
\[
  \nmB{\sum_{n=1}^N \phi_n \otimes e_n}_{\gamma(U,E)} := \Big(
  \mathbb E \nmB{ \sum_{n=1}^N \gamma_n e_n }_E^2
  \Big)^{1/2}
\]
for all $ N \in \mathbb N_{>0} $, all orthonormal systems
$ (\phi_n)_{n=1}^N $ of $ U $, all sequences $ (e_n)_{n=1}^N $ in $ E $,
and all sequences $ (\gamma_n)_{n=1}^N $ of independent standard Gaussian
random variables. Here, $ \mathbb E $ denotes the expectation operator
associated with the probability space on which $ \gamma_1, \ldots, \gamma_N $
are defined.
It is noteworthy that when $ E $ is a Hilbert space, $ \gamma(U,E) $ is identical to the space of
all Hilbert-Schmidt operators from $ U $ to $ E $, equipped with the same norm.
Furthermore, of particular significance is the isometric isomorphism between $ L^q(\mathcal O;H) $ and
$ \gamma(H,L^q(\mathcal O)) $ for any $ q \in [1,\infty) $.
For a comprehensive study of $\gamma$-radonifying operators, the reader is directed
to \cite[Chapter~9]{HytonenWeis2017}.

\medskip\noindent\textbf{$ \mathcal R $-boundedness}. For any two Banach spaces
$ E_1 $ and $ E_2 $, an operator family $ \mathcal A \subset \mathcal L(E_1,
E_2) $ is said to be $ \mathcal R $-bounded if there exists a constant $ C > 0 $
such that
\[
  \int_0^1 \nmB{ \sum_{n=1}^N r_n(t) B_n x_n  }_{E_2}^2 \, \mathrm{d}t
  \leqslant C \int_0^1 \nmB{ \sum_{n=1}^N r_n(t) x_n }_{E_1}^2 \, \mathrm{d}t
\]
for all $ N \geqslant 1 $, all sequences $ (B_n)_{n=1}^N $ in $\mathcal A $,
all sequences $ (x_n)_{n=1}^N $ in $ E_1 $, and all sequences $ (r_n)_{n=1}^N $
of independent symmetric $ \{-1,1\} $-valued random variables on $ [0,1] $.
We denote by $ \mathcal R(\mathcal A) $ the infimum of these $ C $'s.
For a comprehensive treatment of $\mathcal R $-boundedness, the reader is referred to
\cite[Chapter~10]{HytonenWeis2017} and \cite[Chapter~4]{Pruss2016}.

\medskip\noindent\textbf{$ H^\infty $-calculus}.
A sectorial operator \( A \) with an angle of analyticity \( \omega_A \) on a Banach space \( E \) is said to possess a bounded \( H^\infty \)-calculus if there exists \( \sigma \in (\omega_A, \pi] \) such that
\[
  \nmB{
    \int_{\partial\Sigma_{\sigma}} \varphi(z) (z-A)^{-1} \, \mathrm{d}z
  }_{\mathcal L(E)} \leqslant
  C \sup_{z \in \Sigma_{\sigma}} \snm{\varphi(z)}
\]
holds for all \( \varphi \in \mathcal H_0^\infty(\Sigma_\sigma) \), where \( C \) is a positive constant that is independent of \( \varphi \). The space \( \mathcal H_0^\infty(\Sigma_\sigma) \) is defined as
\begin{align*}
  \mathcal H_0^\infty(\Sigma_\sigma)
  &:= \left\{
    \varphi: \Sigma_{\sigma} \to \mathbb C \mid \,
    \text{the function $\varphi$ is analytic and there exists $ \varepsilon > 0 $ such that} \right. \\
  & \qquad\qquad\qquad\qquad\qquad
  \left. \sup_{z \in \Sigma_\sigma}
  \left( \frac{1+\snm{z}^2}{\snm{z}} \right)^\varepsilon
  \snm{\varphi(z)} < \infty
  \right\}.
\end{align*}
The infimum of all such \( \sigma \) values is referred to as the angle of the \( H^\infty \)-calculus of \( A \).
It is worth mentioning that a wide array of partial differential operators admits a bounded \( H^\infty \)-calculus, including the negative Laplacian and the negative Stokes operator; see Section~9 of \cite{Kalton2006}. For an exhaustive treatment of the \( H^\infty \)-calculus, the reader is directed to Chapter 5 of \cite{Haase2006} and Chapter 10 of \cite{HytonenWeis2017}.

\medskip\noindent\textbf{Stochastic integral}.
Assume that $(\Omega, \mathcal{F}, \mathbb{P})$ is a given complete probability
space equipped with a right-continuous filtration $\mathbb{F} := (\mathcal{F}_t)_{t \geqslant 0}$.
On this space, we are given a sequence of independent $\mathbb{F}$-adapted Brownian motions
$(\beta_n)_{n \in \mathbb{N}}$ such that for any $ 0 \leqslant s < t < \infty $ and
for any $ n \in \mathbb N $, the increment $ \beta_n(t) - \beta_n(s) $ is independent of $ \mathcal F_s$.
In the sequel, we will
use $ \mathbb E $ to denote the expectation operator associated with the
probability space $ (\Omega, \mathcal F,\mathbb P) $.
Let $H$ be a separable Hilbert space with inner product
$(\cdot, \cdot)_H$ and an orthonormal basis $(h_n)_{n \in \mathbb{N}}$.
The $ \mathbb F $-adapted $H$-cylindrical Brownian motion
$W$ is defined such that for each $t \geqslant 0$, $W(t)$ is an element of $\mathcal{L}(H, L^2(\Omega))$,
given explicitly by
\[
  W(t)h = \sum_{n \in \mathbb{N}} (h, h_n)_H \beta_n(t), \quad \forall h \in H.
\]
The reader is referred to \cite{Neerven2007} for the theory of stochastic integrals with respect to $W$
in UMD Banach spaces. 
For any $ p,q \in (1,\infty) $, let
$ L_\mathbb F^p(\Omega;L^q(\mathcal O;L^2(\mathbb R_{+};H))) $
denote the space of all $ \mathbb F $-adapted $ L^q(\mathcal O;H) $-valued processes in
$ L^p(\Omega;L^q(\mathcal O;L^2(\mathbb R_{+};H))) $.
The stochastic integral has the following essential isomorphism feature;
see, e.g., \cite[Theorem 2.3]{Neerven2012b}.
\begin{lemma}
  \label{lem:integral}
  For any $ p,q \in (1,\infty) $, there exist two positive constants
  $ c_0 $ and $ c_1 $ such that
  \begin{equation} 
    c_0 \mathbb E\norm{f}_{L^q(\mathcal O;L^2(\mathbb R_{+};H))}^p
    \leqslant \mathbb E\nmB{
      \int_{\mathbb R_{+}} f(t) \,\mathrm{d}W(t)
    }_{L^q(\mathcal O)}^p \leqslant
    c_1 \mathbb E\norm{f}_{L^q(\mathcal O;L^2(\mathbb R_{+};H))}^p
  \end{equation}
  for all $ f \in L_\mathbb F^p(\Omega;L^q(\mathcal O;L^2(\mathbb R_{+};H))) $.
\end{lemma}

\medskip\noindent\textbf{Discrete spaces}.
For any Banach space $ E $ and $ p \in [1,\infty) $, define
\[ 
  \ell^p(E) := \Big\{
    (v_j)_{j \in \mathbb N}  \Bigl| \,\,
    \sum_{j \in \mathbb N} \norm{v_j}_E^p < \infty
  \Big\},
\]
and endow this space with the norm
\[ 
  \norm{(v_j)_{j\in\mathbb N}}_{\ell^p(E)} :=
  \Big( \sum_{j \in \mathbb N} \norm{v_j}_E^p \Big)^{1/p}
  \quad\text{for all } (v_j)_{j \in \mathbb N} \in \ell^p(E).
\]
For any $ v \in \ell^p(E) $, we use $ v_j $, $ j \in \mathbb N $, to denote
its $ j $-th element.

\section{Stability estimates}
\label{sec:stability}
Let \(0 < \tau < 1\) be a fixed time step. For any \(p, q, r \in [1, \infty)\), let \(\ell_{\mathbb{F}}^p(L^r(\Omega; L^q(\mathcal{O}; H)))\) denote the space
\[
\left\{ v \in \ell^p(L^r(\Omega; L^q(\mathcal{O}; H))) \mid v_j \text{ is } \mathcal{F}_{j\tau}\text{-measurable for all } j \in \mathbb{N} \right\}.
\]
It is well-known that \(\ell_{\mathbb{F}}^p(L^r(\Omega; L^q(\mathcal{O}; H)))\) forms a Banach space when endowed with the norm \(\|\cdot\|_{\ell^p(L^r(\Omega; L^q(\mathcal{O}; H)))}\).
For each \(j \in \mathbb{N}\) and for any \(v \in L^p(\Omega; L^q(\mathcal{O}; H))\) with \(p, q \in [2, \infty)\) that is \(\mathcal{F}_{j\tau}\)-measurable, we introduce the shorthand notation \(v \delta W_j\) to represent the stochastic integral
\[
\int_{j\tau}^{j\tau+\tau} v \, \mathrm{d}W(t).
\]
This notation is adopted for the sake of brevity and clarity.

This section studies the stability of the following Euler scheme: seek $ Y := (Y_j)_{j\in\mathbb N} $ such that
\begin{subequations} 
  \label{eq:Y-def}
  \begin{numcases}{}
    Y_{j+1} - Y_j + \tau A Y_{j+1} = f_j \delta W_j, \quad j \in \mathbb N, \\
    Y_0 = 0,
  \end{numcases}
\end{subequations}
where the sequence $ f := (f_j)_{j\in \mathbb N} $ is given.

The main result of this section are the following two theorems.
\begin{theorem} 
  \label{thm:time-regu}
  Let $ p,q,r \in (1,\infty) $. Assume that $ A $ is a densely defined
  sectorial operator on $ L^q(\mathcal O) $ and
  \(\{z(z-A)^{-1} \mid z \in \mathbb C\setminus\overline{\Sigma_{\theta_A}}\} \)
  is $ \mathcal R $-bounded in $ \mathcal L(L^q(\mathcal O)) $,
  where $ \theta_A \in (0,\pi/2) $.
  Let $ Y $ be the solution to the discretization \cref{eq:Y-def} with
  \[
    f \in \ell_{\mathbb F}^p(L^r(\Omega; L^q(\mathcal O;H))).
  \]
  Then the following stability estimate holds:
  \begin{equation}
    \label{eq:time-regu}
    \left[
      \sum_{j=0}^\infty  \nmB{
        \frac{Y_{j+1} - Y_j}{\sqrt\tau}
      }_{L^r(\Omega;L^q(\mathcal O))}^p
    \right]^{1/p} \leqslant
    c \norm{f}_{\ell^p(L^r(\Omega;L^q(\mathcal O;H)))},
  \end{equation}
  where $ c $ is a constant independent of the time step $ \tau $.
\end{theorem}

\begin{theorem} 
  \label{thm:space-regu}
  Let $ p \in (2,\infty) $ and $ q \in [2,\infty) $. Assume that $ A $ is a
  densely defined sectorial operator on $ L^q(\mathcal O) $ satisfying the
  following conditions:
  \begin{itemize}
    \item $ A $ has a dense range in $ L^q(\mathcal O) $;
    \item  there exists $ \theta_A \in (0,\pi/2) $ such that
      \begin{equation}
        \label{eq:z-A-inv}
        \sup_{z \in \mathbb C \setminus \{0\}, \, \snm{\operatorname{Arg} z} \geqslant \theta_A} \,
        \snm{z} \norm{(z-A)^{-1}}_{\mathcal L(L^q(\mathcal O))}
        < \infty;
      \end{equation}
    \item $ A $ admits a bounded $ H^\infty $-calculus of angle less than $ \theta_A $.
  \end{itemize}
  Let $ Y $ be the solution to the discretization \cref{eq:Y-def} with
  \[
    f \in \ell_{\mathbb F}^p(L^p(\Omega; L^q(\mathcal O;H))).
  \]
  Then the following discrete stochastic maximal $ L^p $-regularity estimate holds:
  \begin{equation}
    \label{eq:space-regu}
    \norm{A^{1/2}Y}_{\ell^p(L^p(\Omega;L^q(\mathcal O)))} \leqslant
    c \norm{f}_{\ell^p(L^p(\Omega;L^q(\mathcal O;H)))},
  \end{equation}
  where $ c $ is a constant independent of $ \tau $.
\end{theorem}

\begin{remark}
  \label{rem:parabolic}
  The discrete stochastic maximal $ L^p $-regularity estimate established in \cref{thm:space-regu}
  is useful for the numerical analysis of significant nonlinear SPDEs, including the stochastic Allen-Cahn
  equations and the stochastic Navier-Stokes equations.
  Here, we present a concise application of Theorem~\ref{thm:space-regu} to the numerical analysis of nonlinear
  stochastic parabolic equations as follows. Let \( T > 0 \) be a fixed time horizon, and let
  \( \mathcal{O} \subset \mathbb{R}^3 \) be a bounded domain with a sufficiently smooth boundary
  \( \partial \mathcal{O} \). Consider a sequence \( (f_n)_{n \in \mathbb{N}} \) of continuous functions from \( \mathcal{O} \times \mathbb{R} \) to \( \mathbb{R} \) that satisfy the following growth condition for each \( n \in \mathbb{N} \):
  \[
    |f_n(x, y)|\leqslant C_F(1 + |y|)
    \quad\text{for all $x \in \mathcal{O} $ and $ y \in \mathbb{R}$},
  \]
  where \( C_F \) is a constant independent of \( n \). Additionally, let \( (\lambda_n)_{n \in \mathbb{N}} \) be a sequence of non-negative real numbers such that
  \[
    \sum_{n \in \mathbb{N}} \lambda_n < \infty.
  \]
  We investigate the following nonlinear stochastic parabolic equation:
  \[
    \begin{cases}
      \mathrm{d}y(t,x) - \Delta y(t,x) \, \mathrm{d}t = \sum_{n \in \mathbb{N}} \sqrt{\lambda_n} f_n(x, y(t,x)) \, \mathrm{d}\beta_n(t), & t \in [0, T], \, x \in \mathcal{O}, \\
      y(t,x) = 0, & t \in (0, T], \, x \in \partial \mathcal{O}, \\
      y(0,x) = 0, & x \in \mathcal{O},
    \end{cases}
  \]
  where \( (\beta_n)_{n \in \mathbb{N}} \) represents a sequence of independent standard Brownian motions as defined in Section~\ref{sec:pre}.
  It is well-known that the negative Laplace operator with homogeneous Dirichlet boundary conditions meets
  the conditions of Theorem~\ref{thm:space-regu}; see, for instance, \cite[Proposition~9.8]{Kalton2006}.
  Let \( J \) be a positive integer and set \( \tau = \frac{T}{J} \). We consider the following temporal semidiscretization:
  \[
    \begin{cases}
      Y_{j+1}(x) - Y_j(x) - \tau \Delta Y_{j+1}(x) = \sum_{n \in \mathbb{N}} \sqrt{\lambda_n}
      \int_{j\tau}^{j\tau+\tau} f_n(x, Y_j(x)) \, \mathrm{d}\beta_n(t), & 0\leqslant j < J, \, x \in \mathcal{O}, \\
      Y_j(x) = 0, & 1\leqslant j\leqslant J, \, x \in \partial \mathcal{O}, \\
      Y_0(x) = 0, & x \in \mathcal{O}.
    \end{cases}
  \]
  For brevity, we henceforth refer to a quantity as ``bounded" to imply uniform boundedness with respect to \( \tau \).
  For any $ p \in (2,\infty) $, a routine computation reveals that
  $ \mathbb{E} \big[ \tau\sum_{j=1}^J \|Y_j\|_{L^2(\mathcal{O})}^p \big] $ is bounded,
  which implies by Theorem~\ref{thm:space-regu} that
  $ \mathbb{E} \big[ \tau \sum_{j=1}^J \|\nabla Y_j\|_{L^2(\mathcal{O})}^p \big] $ is also bounded.
  Applying Sobolev's embedding theorem, we further obtain that
  $ \mathbb{E} \big[ \tau \sum_{j=1}^J \|Y_j\|_{L^6(\mathcal{O})}^p \big] $ is
  bounded. Hence, invoking Theorem~\ref{thm:space-regu} once more yields that $ \mathbb{E} \big[\tau \sum_{j=1}^J \|\nabla Y_j\|_{L^6(\mathcal{O})}^p\big] $
  is bounded.
  By iterating this argument, we conclude that
  \[
    \mathbb{E}\left[ \tau \sum_{j=1}^J \|\nabla Y_j\|_{L^q(\mathcal{O})}^p \right]
    \quad\text{is bounded for all $ p,q \in (2,\infty)$}.
  \]
  This methodology was utilized in \cite{LiZhou2024-time}; interested readers are referred there for additional details.
  Furthermore, \cref{thm:time-regu} implies that
  \begin{align*}
    \mathbb E \left[
      \tau\sum_{j=0}^{J-1} \left\lVert \frac{Y_{j+1}-Y_j}{\sqrt\tau}\right\rVert_{L^q(\mathcal O)}^p
    \right] \quad\text{is bounded for all $ p,q \in (2,\infty)$}.
  \end{align*}
  Analogous stability results can be straightforwardly derived when the initial value is non-zero.
\end{remark}

\subsection{Proof of \texorpdfstring{\cref{thm:time-regu}}{}}
Let $\widetilde{A}$ be the natural extension of $A$ in $L^r(\Omega; L^q(\mathcal{O}))$.
It is straightforward to verify that $\widetilde{A}$ is a sectorial operator on $L^r(\Omega; L^q(\mathcal{O}))$.
Let $(r_n)_{n=1}^\infty$ be an arbitrary sequence of independent symmetric $\{-1, 1\}$-valued random variables on $[0, 1]$. For any $N \geqslant 1$, $(z_n)_{n=1}^N \subset \mathbb{C} \setminus \overline{\Sigma_{\theta_A}}$, and $(v_n)_{n=1}^N \subset L^r(\Omega; L^q(\mathcal{O}))$, we have
\begin{align*}
    & \left[
      \int_0^1 \Bigl\|
        \sum_{n=1}^N r_n(t) z_n(z_n - \widetilde{A})^{-1}v_n
      \Bigr\|_{L^r(\Omega; L^q(\mathcal{O}))}^2 \, \mathrm{d}t
    \right]^{1/2} \\
  \stackrel{\text{(i)}}{\leqslant} &
  c \left[
    \int_0^1 \Bigl\|
      \sum_{n=1}^N r_n(t) z_n(z_n - \widetilde{A})^{-1}v_n
    \Bigr\|_{L^r(\Omega; L^q(\mathcal{O}))}^r \, \mathrm{d}t
  \right]^{1/r} \\
  ={} &
  c \left[
    \mathbb{E} \int_0^1 \Bigl\|
      \sum_{n=1}^N r_n(t) z_n(z_n - \widetilde{A})^{-1}v_n
    \Bigr\|_{L^q(\mathcal{O})}^r
    \, \mathrm{d}t
  \right]^{1/r} \\
  ={} &
  c \left[
    \mathbb{E} \int_0^1 \Bigl\|
      \sum_{n=1}^N r_n(t) z_n(z_n - A)^{-1}v_n
    \Bigr\|_{L^q(\mathcal{O})}^r
    \, \mathrm{d}t
  \right]^{1/r} \\
  \stackrel{\text{(ii)}}{\leqslant} &
  c \left[
    \mathbb{E} \left(
      \int_0^1
      \Bigl\|
        \sum_{n=1}^N r_n(t) z_n(z_n - A)^{-1}v_n
      \Bigr\|_{L^q(\mathcal{O})}^2
      \, \mathrm{d}t
    \right)^{r/2}
  \right]^{1/r},
\end{align*}
where inequalities (i) and (ii) follow from the Kahane-Khintchine inequality (see, e.g., \cite[Theorem 6.2.4]{HytonenWeis2017}).
By the $\mathcal{R}$-boundedness of the set $\{z(z - A)^{-1} \mid z \in \mathbb{C} \setminus \overline{\Sigma_{\theta_A}}\}$, we further deduce that
\begin{align*}
    & \left[
      \int_0^1 \Bigl\|
        \sum_{n=1}^N r_n(t) z_n(z_n - \widetilde{A})^{-1}v_n
      \Bigr\|_{L^r(\Omega; L^q(\mathcal{O}))}^2 \, \mathrm{d}t
    \right]^{1/2} \\
    \leqslant{}
    & c \mathcal{R}\left(\{z(z - A)^{-1} \mid z \in \mathbb{C} \setminus \overline{\Sigma_{\theta_A}}\}\right)
    \left[
      \mathbb{E} \left(
        \int_0^1
        \Bigl\|
          \sum_{n=1}^N r_n(t) v_n
        \Bigr\|_{L^q(\mathcal{O})}^2
        \, \mathrm{d}t
      \right)^{r/2}
    \right]^{1/r}.
\end{align*}
Applying the Kahane-Khintchine inequality once more, we obtain
\begin{align*}
     \left[
      \mathbb{E} \left(
        \int_0^1
        \Bigl\|
          \sum_{n=1}^N r_n(t) v_n
        \Bigr\|_{L^q(\mathcal{O})}^2
        \, \mathrm{d}t
      \right)^{r/2}
    \right]^{1/r} 
     & \leqslant
    c \left[
      \mathbb{E} \int_0^1 \Bigl\|
        \sum_{n=1}^N r_n(t) v_n
      \Bigr\|_{L^q(\mathcal{O})}^r
      \, \mathrm{d}t
    \right]^{1/r} \\
    &=
    c \left[
      \int_0^1 \Bigl\|
        \sum_{n=1}^N r_n(t) v_n
      \Bigr\|_{L^r(\Omega; L^q(\mathcal{O}))}^r
      \, \mathrm{d}t
    \right]^{1/r} \\
    & \leqslant
    c \left[
      \int_0^1 \Bigl\|
        \sum_{n=1}^N r_n(t) v_n
      \Bigr\|_{L^r(\Omega; L^q(\mathcal{O}))}^2
      \, \mathrm{d}t
    \right]^{1/2}.
\end{align*}
Combining these results, we conclude that
\begin{align*}
    & \left[
      \int_0^1 \Bigl\|
        \sum_{n=1}^N r_n(t) z_n(z_n - \widetilde{A})^{-1}v_n
      \Bigr\|_{L^r(\Omega; L^q(\mathcal{O}))}^2 \, \mathrm{d}t
    \right]^{1/2} \\
    \leqslant{}
    & c \mathcal{R}\big(\{z(z - A)^{-1} \mid z \in \mathbb{C} \setminus \overline{\Sigma_{\theta_A}}\}\big)
    \left[
      \int_0^1 \Bigl\|
        \sum_{n=1}^N r_n(t) v_n
      \Bigr\|_{L^r(\Omega; L^q(\mathcal{O}))}^2
      \, \mathrm{d}t
    \right]^{1/2}.
\end{align*}
It follows that $ \{z(z-\widetilde A)^{-1} \mid z \in \mathbb C \setminus\overline{\Sigma_{\theta_A}}\} $
is $ \mathcal R $-bounded in $ \mathcal L(L^r(\Omega;L^q(\mathcal O))) $. Moreover,
\cite[Proposition~4.2.15]{HytonenWeis2016} implies that $
L^r(\Omega;L^q(\mathcal O)) $ is a UMD space. Therefore,
we use \cite[Theorem 3.2]{Kemmochi2016} to conclude that
\begin{align*}
  \left[
    \sum_{j=0}^\infty \nmB{
      \frac{Y_{j+1} - Y_j
    }\tau }_{L^r(\Omega;L^q(\mathcal O))}^p
  \right]^{1/p} \leqslant c \left[
    \sum_{j=0}^\infty \nmB{
      \frac{f_j}\tau \delta W_j
    }_{L^r(\Omega;L^q(\mathcal O))}^p
  \right]^{1/p},
\end{align*}
which implies
\begin{align*} 
  \left[
    \sum_{j=0}^\infty
    \nmB{ \frac{Y_{j+1} - Y_j }{\sqrt\tau}}_{L^r(\Omega;L^q(\mathcal O))}^p
  \right]^{1/p} \leqslant c \left[
    \sum_{j=0}^\infty \nmB{
      \frac{f_j}{\sqrt\tau} \delta W_j
    }_{L^r(\Omega;L^q(\mathcal O))}^p
  \right]^{1/p}.
\end{align*}
Consequently, the desired inequality \cref{eq:time-regu} follows from the
estimate
\begin{align*}
  \bigg[
    \sum_{j=0}^\infty \nmB{
      \frac{f_j}{\sqrt\tau} \delta W_j
    }_{L^r(\Omega;L^q(\mathcal O))}^p
  \bigg]^{1/p} 
  &\leqslant
  c \bigg[
    \sum_{j=0}^\infty
    \norm{f_j}_{L^r(\Omega;L^q(\mathcal O;H))}^p
  \bigg]^{1/p} \quad\text{(by \cref{lem:integral})} \\
  &= c \norm{f}_{\ell^p(L^r(\Omega;L^q(\mathcal O;H)))}.
\end{align*}
This completes the proof of \cref{thm:time-regu}.

\subsection{Proof of \texorpdfstring{\cref{thm:space-regu}}{}}
\label{ssec:space-regu}
Throughout this subsection, we will assume that the conditions in
\cref{thm:space-regu} are always satisfied. Firstly, let us introduce
some notations. Let $ A^* $ be the dual operator of $ A $ on $ L^{q'}(\mathcal O) $.
It is standard that $ A^* $ is a sectorial operator on $ L^{q'}(\mathcal O) $;
see, e.g., \cite[Theorem~2.4.1]{Sinha2017book}.
Moreover, $ A^* $ has a bounded $ H^\infty $-calculus 
as $ A $; see \cite[Proposition 10.2.20]{HytonenWeis2017}.

\begin{figure}[ht]
  \centering
  \includegraphics[width=0.5\textwidth]{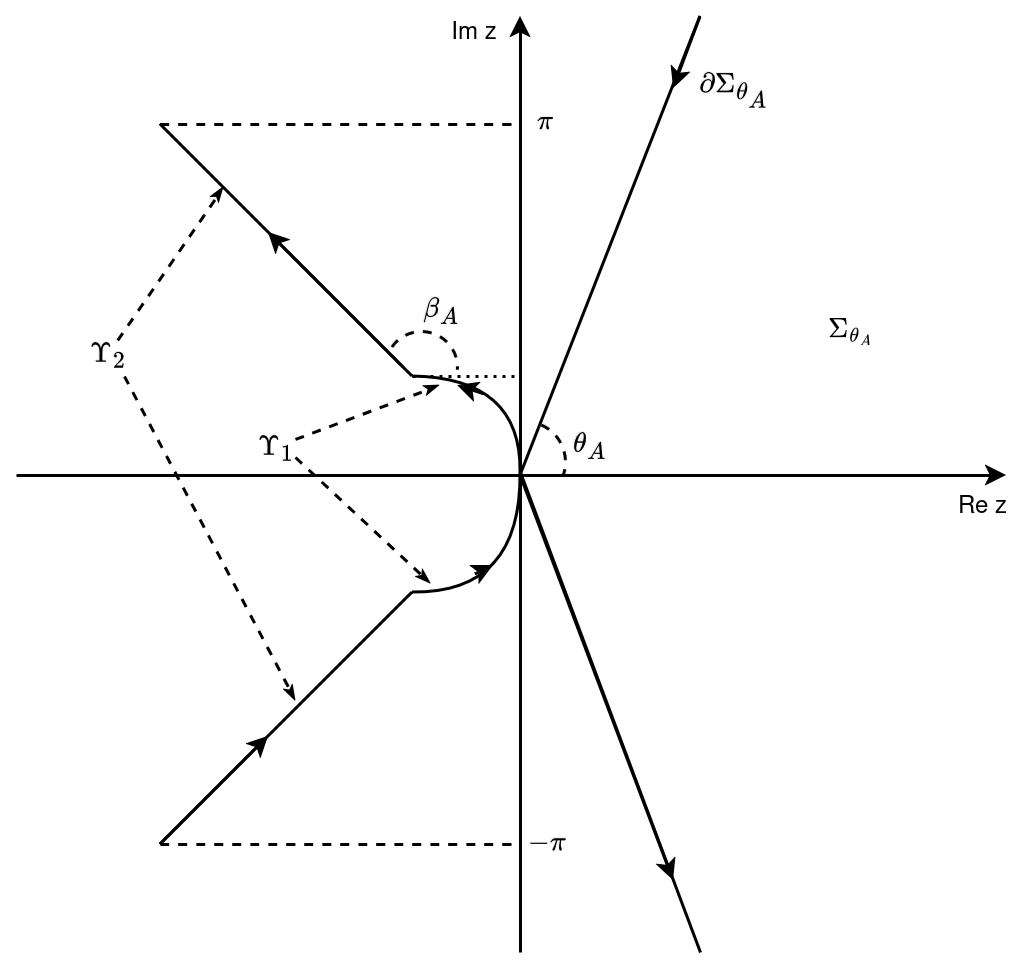}
  \caption{The orientations of $ \Upsilon_1 $, $ \Upsilon_2 $ and $ \partial\Sigma_{\theta_A} $.}
  \label{fig:upsilon_sigma}
\end{figure}

Suppose the complex logarithm, denoted by \( \log \), is confined to the
strip \( \{z \in \mathbb{C} \mid -\pi < \operatorname{Im} z \leqslant \pi\} \).
Choose any \( 0 < \epsilon_A < \cot\theta_A \). 
A simple calculation gives the existence of a positive real number \( k_A \)
such that
\[
e^{-x} \cos y - 1 \leqslant \epsilon_A e^{-x} |\sin y|
\]
holds for all \( x \geqslant -k_A |y| \) with \( y \in [-\pi, \pi] \).
This leads to
\[
  \operatorname{Re}(e^{-(x+iy)}-1) \leqslant
  \epsilon_A |\operatorname{Im}(e^{-(x+iy)}-1)|
\]
for the same conditions on \( x \) and \( y \). Consequently,
\begin{equation}
  \label{eq:lost}
  |\operatorname{Arg}(e^{-z}-1)| \geqslant \operatorname{arccot}\epsilon_A > \theta_A
  \quad\text{for all $ z \in \Sigma_{\beta_A} $ with $ -\pi \leqslant \operatorname{Im} z \leqslant \pi $},
\end{equation}
where \( \beta_A := \pi/2 + \arctan k_A \).
By selecting \( \alpha_A \) sufficiently large within the interval
\( (\theta_A,\pi/2) \), we can establish the existence of \( \alpha_A \)
such that the curve described by \( \{-\log(1+re^{i\alpha_A}) \mid r > 0\} \)
intersects the boundary \( \partial\Sigma_{\beta_A} \)
at a unique point. We then define
\begin{align*}
  \Upsilon
  &:= \Upsilon_1 \cup \Upsilon_2, \\
  \Upsilon_1
  &:= \Big\{
    -\log(1 + re^{i\alpha_A}) \big| \, 0 \leqslant r \leqslant r_A
  \Big\} \bigcup \Big\{
    -\log(1 + re^{-i\alpha_A}) \big| \, 0 \leqslant r \leqslant r_A
  \Big\}, \\
  \Upsilon_2
  &:= \left\{
    re^{-i\beta_A} \Big| \, |\log(1+r_Ae^{i\alpha_A})| \leqslant r \leqslant \frac{\pi}{\cos(\beta_A-\pi/2)}
  \right\} \\
  & \qquad \bigcup \left\{
    re^{i\beta_A} \Big| \, |\log(1+r_Ae^{i\alpha_A})| \leqslant r \leqslant \frac{\pi}{\cos(\beta_A-\pi/2)}
  \right\},
\end{align*}
where \( r_A \) is the positive number for which \( -\log(1+r_Ae^{i\alpha_A}) \) corresponds to the aforementioned intersection point.
Given the compactness of \( \{e^{-z}-1 \mid z \in \Upsilon_2\} \) and its
non-intersection with \( \partial\Sigma_{\theta_A} \), it follows that
\begin{equation}
  \label{eq:upsilon2}
  \inf_{z \in \Upsilon_2} \inf_{\lambda \in \Sigma_{\theta_A}} |e^{-z} - 1 - \lambda| > 0.
\end{equation}
In addition, it is easily verified that
\begin{equation}
  \label{eq:auxi-bound}
  \sup_{z \in \Upsilon\setminus\{0\}}
  \frac{\snm{e^{-z}-1}}{\ln\snm{e^{-z}}} < \infty.
\end{equation}
For any $ z \in \Sigma_{\theta_A} $, define
\begin{align}
  \varphi_{+}(z) &:= (2\pi i)^{-1/2} e^{i\alpha_A/4} z^{1/4}
  (-e^{i\alpha_A} + z)^{-1/2},
  \label{eq:varphi+} \\
  \varphi_{+}^*(z) &:= (-2\pi i)^{-1/2} e^{-i\alpha_A/4}
  z^{1/4} (-e^{-i\alpha_A} + z)^{-1/2},
  \label{eq:varphi+*} \\
  \varphi_{-}(z) &:= (-2\pi i)^{-1/2} e^{-i\alpha_A/4}
  z^{1/4} (-e^{-i\alpha_A} + z)^{-1/2}, 
  \label{eq:varphi-} \\
  \varphi_{-}^*(z) &:= (2\pi i)^{-1/2} e^{i\alpha_A/4}
  z^{1/4} (-e^{i\alpha_A} + z)^{-1/2}.
  \label{eq:vaprhi-*}
\end{align}
It is evident that the above four functions all belong to $ \mathcal H_0^\infty(\Sigma_{\theta_A}) $.
For any \( z \in \Upsilon \setminus \{0\} \),
two additional functions, \( \Psi \) and \( \Psi^* \), are defined:
\begin{align}
  \Psi(z) &:= (2\pi i)^{-1/2} (e^{-z}-1)^{-1/4} (\tau A)^{1/4} (1 - e^{-z} + \tau A)^{-1/2}, \label{eq:varphi} \\
  \Psi^*(z) &:= (-2\pi i)^{-1/2} (e^{-\bar z}-1)^{-1/4} (\tau A^*)^{1/4} (1 - e^{-\bar z} + \tau A^*)^{-1/2}.
  \label{eq:varphi*}
\end{align}
Lastly, we introduce a family of operators, \( \{\mathcal I(z) \mid z \in \Upsilon\} \),
which map between \( \ell_{\mathbb F}^p(L^p(\Omega;L^q(\mathcal O;H))) \)
and \( \ell^p(L^p(\Omega;L^q(\mathcal O))) \), as follows: for any $ z \in \Upsilon $ and
$ g \in \ell_{\mathbb F}^p(L^p(\Omega;L^q(\mathcal O;H))) $, 
$ \mathcal I(z)g $ is defined by
\begin{subequations}
  \label{eq:I-def}
  \begin{numcases}{}
    \big( \mathcal I(z)g \big)_0 := 0, \\
    \big( \mathcal I(z)g \big)_j := \frac1{\sqrt\tau}\sum_{k=0}^{j-1} 
    (e^{-z} - 1)^{1/2} e^{(j-k-1)z} g_k \delta W_k,
    \quad j \geqslant 1.
  \end{numcases}
\end{subequations}
In \cref{sec:some_proofs} we demonstrate that the operators
$ \mathcal I(z) $, $ z \in \Upsilon $, are indeed bounded linear
operators from $ \ell_{\mathbb F}^p(L^p(\Omega;L^q(\mathcal O;H))) $ to
$ \ell^p(L^p(\Omega;L^q(\mathcal O))) $.

Secondly, let us introduce the $ \mathcal R $-boundedness of $ \{\mathcal I(z) \mid z \in \Upsilon \}  $ and
the square function bounds associated with $ \varphi_{+} $, $ \varphi_{+}^* $, $ \varphi_{-} $, 
$ \varphi_{-}^* $, $ \Psi $ and $ \Psi^* $.
For the clarity of the presentation of the main idea of the proof of
\cref{thm:space-regu}, we put some technical lemmas in \cref{sec:some_proofs}.

\begin{lemma}
  \label{lem:I-R-bounded}
  $ \mathcal R(\{\mathcal I(z) \mid z \in \Upsilon\}) $
  is uniformly bounded with respect to the time step $ \tau $.
\end{lemma}
\begin{proof}
  Let $ \{\Pi_m \mid m \in \mathbb N \} $ be defined by
  \cref{eq:calII-def}.
  For any $ z \in \Upsilon\setminus\{0\} $ and
  $ g \in \ell_{\mathbb F}^{p}(L^p(\Omega;L^q(\mathcal O;H))) $,
  by \cref{eq:I-def} and the identity
  \(e^{(j-k-1)z} = (1-e^z) \sum_{m=j-k-1}^\infty e^{mz} \),
  a direct calculation gives
  \begin{align*} 
    \big(\mathcal I(z)g\big)_j
    = \sum_{m=0}^\infty \sqrt{1+m} \, (e^{-z}-1)^{3/2}
    e^{(m+1)z} (\Pi_{m}g)_j, \quad \forall j \geqslant 1.
  \end{align*}
  Considering further that $ (\mathcal I(z)g)_0 = 0 $
  for all $ z \in \Upsilon\setminus\{0\} $ and
  $ (\Pi_mg)_0 = 0 $ for each $ m \in \mathbb N $, we arrive
  at the representation
  \begin{equation}
    \label{eq:I-IIm}
    \mathcal I(z)  = \sum_{m=0}^\infty
    \sqrt{1+m} \, (e^{-z}-1)^{3/2} e^{(1+m)z} \,
    \Pi_{m}, \quad \forall z \in \Upsilon\setminus\{0\}.
  \end{equation}
  For any $ z \in \Upsilon \setminus \{0\} $, an elementary calculation using the fact $ \operatorname{Re} z < 0 $ yields
  \begin{align*} 
    \sum_{m=0}^\infty
    \sqrt{1+m} \, \snmb{ (e^{-z}-1)^{3/2} e^{(1+m)z} } 
    \leqslant \snm{e^{-z}-1}^{3/2}
    \int_0^\infty \sqrt{x+1} \, \snm{e^z}^{x} \, \mathrm{d}x.
  \end{align*}
  Upon introducing the change of variable
  $ y = -x\ln\snm{e^z} $, we further obtain
  \begin{align*} 
     \sum_{m=0}^\infty
    \sqrt{1+m} \, \snmb{ (e^{-z}-1)^{3/2} e^{(1+m)z} }
    \leqslant
    \frac{\snm{e^{-z}-1}}{\ln\snm{e^{-z}}}
    \int_0^\infty \sqrt{
      \snm{e^{-z}-1}+ \frac{\snm{e^{-z}-1}}{\ln\snm{e^{-z}}} y 
    } \, \, e^{-y} \, \mathrm{d}y.
  \end{align*}
  Hence, from the uniform boundedness established in \cref{eq:auxi-bound}, we deduce that the supremum
  \begin{align*} 
    \mathcal M := \sup_{z \in \Upsilon \setminus\{0\}} \, \sum_{m=0}^\infty
    \sqrt{1+m} \, \snmb{ (e^{-z}-1)^{3/2} e^{(1+m)z} } 
  \end{align*}
  is finite. By Kahane's contraction principle (see \cite[Proposition~6.1.13]{HytonenWeis2017})
  and \cref{lem:Pi_m}, we infer that the $ \mathcal R $-bound of the set
  \[
    \left\{e^{i\theta}
      \mathcal M \Pi_m \mid z \in \Upsilon, \, m \in \mathbb N, \, \theta \in (-\pi,\pi]
    \right\}
  \]
  remains uniformly bounded with respect to $ \tau $.
  Given that the set $ \{\mathcal I(z) \mid z \in \Upsilon\} $,
  according to \cref{eq:I-IIm} and acknowledging $ \mathcal I(0) = 0 $,
  resides within the absolute convex hull of 
  the aforementioned set, we use the convexity of $ \mathcal R$-bounds (see \cite[Proposition~8.1.21]{HytonenWeis2017}) to conclude that
  the $ \mathcal R$-bound of \( \{\mathcal I(z) \mid z \in \Upsilon\} \) is also uniformly
  bounded with respect to the time step $ \tau $.
  This completes the proof.
\end{proof}

By the square function bounds established in \cite[Theorem~10.4.16]{HytonenWeis2017},
in conjunction with the properties of cotype as stated in \cite[Proposition~7.1.4]{HytonenWeis2017},
we obtain the following square function estimates for the functions $ \varphi_{-} $, $ \varphi_{+} $, $ \varphi_{-}^* $, and $ \varphi_{+}^* $.

\begin{lemma}
  \label{lem:varphi+-}
  For any $ g \in \ell^p(L^p(\Omega;L^q(\mathcal O;H))) $, we have
  \begin{align}
    \norm{\varphi_{-}(rA)g}_{
      \gamma(
      L^2(\mathbb R_{+},\frac{\mathrm{d}r}r), \,
      \ell^p(L^p(\Omega;L^q(\mathcal O;H)))
      )
    } \leqslant c \norm{ g }_{
      \ell^p(L^p(\Omega;L^q(\mathcal O;H)))
    }, \label{eq:varphi+A} \\
    \norm{\varphi_{+}(rA)g}_{
      \gamma(
      L^2(\mathbb R_{+},\frac{\mathrm{d}r}r), \,
      \ell^p(L^p(\Omega;L^q(\mathcal O;H)))
      )
    } \leqslant c \norm{ g }_{
      \ell^p(L^p(\Omega;L^q(\mathcal O;H)))
    }. \label{eq:varphi-A}
  \end{align}
\end{lemma}
\begin{lemma} 
  \label{lem:varphi+-*}
  For any $ g \in \ell^{p'}(L^{p'}(\Omega;L^{q'}(\mathcal O))) $, we have
  \begin{align}
    \norm{\varphi_{-}^*(rA^*)g}_{
      \gamma(L^2(\mathbb R_{+}, \frac{\mathrm{d}r}r), \ell^{p'}(L^{p'}(\Omega;L^{q'}
      (\mathcal O))))
    } \leqslant c \norm{g}_{
      \ell^{p'}(L^{p'}(\Omega;L^{q'}(\mathcal O)))
    }, \label{eq:varphi*+A} \\
    \norm{\varphi_{+}^*(rA^*)g}_{
      \gamma(L^2(\mathbb R_{+}, \frac{\mathrm{d}r}r), \ell^{p'}(L^{p'}(\Omega;L^{q'}
      (\mathcal O))))
    } \leqslant c \norm{g}_{
      \ell^{p'}(L^{p'}(\Omega;L^{q'}(\mathcal O)))
    }. \label{eq:varphi*-A}
  \end{align}
\end{lemma}

We proceed to establish a square function estimate for $ \Psi $ as follows.
\begin{lemma}
  \label{lem:varphi}
  For any $ f \in \ell^p(L^p(\Omega;L^q(\mathcal O;H))) $, we have
  \begin{equation}
    \label{eq:varphi-esti}
    \norm{\Psi f}_{
      \gamma(L^2(\Upsilon_2,\snm{\mathrm{d}z}), \ell^p(L^p(\Omega;L^q(\mathcal O;H))))
    } \leqslant c \norm{f}_{\ell^p(L^p(\Omega;L^q(\mathcal O;H)))}.
  \end{equation}
\end{lemma}
\begin{proof}
For brevity, we denote the space $ \ell^p(L^p(\Omega;L^q(\mathcal O;H))) $ by $ X $. Given any $ z \in \Upsilon_2 $, the inequality \cref{eq:upsilon2} implies that
\[
  \left( \cdot \right)^{1/4} \left( 1 - e^{-z} + \tau \cdot \right)^{-1/2}
    \in \mathcal H_0^\infty(\Sigma_{\theta_A}).
\]
According to the theory of the Dunford functional calculus (see, e.g., \cite[Chapter~10]{HytonenWeis2017}), we have
\[
    A^{1/4}(1 - e^{-z} + \tau A)^{-1/2} = \frac{1}{2\pi i}
    \int_{\partial\Sigma_{\theta_A}} \lambda^{1/4}(1 - e^{-z} + \tau \lambda)^{-1/2}
    (\lambda - A)^{-1} \, \mathrm{d}\lambda.
\]
Substituting this expression into \cref{eq:varphi}, we obtain for any $ z \in \Upsilon_2 $:
\begin{align*}
    \Psi(z) &= \frac{1}{(2\pi i)^{3/2}} \int_{\partial\Sigma_{\theta_A}}
    (e^{-z} - 1)^{-1/4} (\tau \lambda)^{1/4} (1 - e^{-z} + \tau \lambda)^{-1/2}
    (\lambda - A)^{-1} \, \mathrm{d}\lambda \\
            &= \int_{\partial\Sigma_{\theta_A}}
            G(z,\lambda) (\tau \lambda)^{1/4}
            (\lambda-A)^{-1} \, \mathrm{d}\lambda,
\end{align*}
where 
\[
    G(z,\lambda) := (2\pi i)^{-3/2}
    (e^{-z}-1)^{-1/4}(1-e^{-z}+\tau\lambda)^{-1/2},
    \quad z \in \Upsilon_2, \, \lambda \in \partial\Sigma_{\theta_A}.
\]
Consequently,
\begin{align*}
    \norm{ \Psi f }_{\gamma(L^2(\Upsilon_2,\snm{\mathrm{d}z}), \, X)}
    &\leqslant \int_{\partial\Sigma_{\theta_A}}
    \nmB{
      G(z,\lambda) (\tau \lambda)^{1/4} 
      (\lambda - A)^{-1} f
    }_{\gamma(L^2(\Upsilon_2,\snm{\mathrm{d}z}), \, X)}
    \, \snm{\mathrm{d}\lambda}.
\end{align*}
Utilizing the $ \gamma $-Fubini isomorphism (\cite[Theorem 9.4.8]{HytonenWeis2017}) and the
isometrical equivalence between $\gamma(L^2(\Upsilon_2,\snm{\mathrm{d}z});H)$ and $ L^2(\Upsilon_2,\snm{\mathrm{d}z};H)$ (\cite[Proposition~9.2.9]{HytonenWeis2017}), we deduce that
\begin{align*}
    \norm{ \Psi f }_{\gamma(L^2(\Upsilon_2,\snm{\mathrm{d}z}), \, X)}
    &\leqslant c \int_{\partial\Sigma_{\theta_A}}
    \snmb{\tau \lambda}^{1/4} 
    \nmb{ (\lambda - A)^{-1} f }_X
    \norm{G(z,\lambda)}_{L^2(\Upsilon_2,\snm{\mathrm{d}z})}
    \, \snm{\mathrm{d}\lambda}.
\end{align*}
Since \cref{eq:upsilon2} ensures that
\[
    \nmb{ G(z,\lambda) }_{L^2(\Upsilon_2,\snm{\mathrm{d}z})} \leqslant \frac{c}{ 1 + \snm{\tau \lambda}^{1/2} }
    \quad \text{for all } \lambda \in \partial\Sigma_{\theta_A},
\]
we conclude that
\begin{align*}
    \norm{ \Psi f }_{\gamma(L^2(\Upsilon_2, \snm{\mathrm{d}z}), \, X)}
    &\leqslant c \int_{\partial\Sigma_{\theta_A}}
    \frac{\snm{\tau \lambda}^{1/4}}{1+\snm{\tau\lambda}^{1/2}}
    \norm{(\lambda - A)^{-1}f}_{X} \, \snm{\mathrm{d}\lambda} \\
    &\leqslant c \int_{\partial\Sigma_{\theta_A}}
    \frac{\snm{\tau\lambda}^{1/4}}{1+\snm{\tau\lambda}^{1/2}}
    \, \frac{\snm{\mathrm{d}\lambda}}{\snm{\lambda}} \norm{f}_X \quad\text{(by \cref{eq:z-A-inv})} \\
    &= c \int_{\partial\Sigma_{\theta_A}}
    \frac{\snm{\eta}^{1/4}}{1+\snm{\eta}^{1/2}}
    \, \frac{\snm{\mathrm{d}\eta}}{\snm{\eta}} \, \norm{f}_X,
\end{align*}
where we have made the change of variable $\eta := \tau\lambda$. Given the convergence of the integral over $\partial\Sigma_{\theta_A}$, we can confirm the validity of the desired estimate \eqref{eq:varphi-esti}, thus completing the proof.
\end{proof}

Similarly, we have the following square function estimate for $ \Psi^* $.
\begin{lemma}
  \label{lem:varphi*}
  For any $ g \in \ell^{p'}(L^{p'}(\Omega;L^{q'}(\mathcal O))) $, we have
  \begin{equation}
    \norm{\Psi^* g}_{
      \gamma(L^2(\Upsilon_2,\snm{\mathrm{d}z}), \, \ell^{p'}(L^{p'}(\Omega;L^{q'}(\mathcal O))))
    } \leqslant c \norm{g}_{\ell^{p'}(L^{p'}(\Omega;L^{q'}(\mathcal O)))}.
  \end{equation}
\end{lemma}

Thirdly, let us introduce a representation formula of the solution to \cref{eq:Y-def}.
For any $ r \geqslant \tau/r_A $, define
\begin{align}
  \mathscr I_{+}(r) &:=
  \mathcal I\left( -\log\left(1 + \frac{\tau e^{i\alpha_A}}r\right) \right), 
  \label{eq:scrI+} \\
  \mathscr I_{-}(r) &:=
  \mathcal I\left( -\log\left(1 + \frac{\tau e^{-i\alpha_A}}r \right) \right).
  \label{eq:scrI-}
\end{align}
\begin{lemma}
  Let $ Y $ be the solution to the discretization \cref{eq:Y-def} with
  \[
    f \in \ell_{\mathbb F}^p(L^p(\Omega;L^q(\mathcal O;H))).
  \]
  Then it holds $\mathbb P $-almost surely that
  \begin{small}
  \begin{equation}
    \label{eq:A12Y}
    A^{1/2} Y = \int_{\tau/r_A}^\infty \Big(
      \frac{
        \varphi_{+}(rA)^2 \mathscr I_{+}(r)
      }{
        r+\tau e^{i\alpha_A}
      } + \frac{
        \varphi_{-}(rA)^2 \mathscr I_{-}(r) 
      }{
        r+\tau e^{-i\alpha_A}
      }
    \Big) f \, \mathrm{d}r + \int_{\Upsilon_2}
    \Psi(z)^2 \mathcal I(z) f \, \mathrm{d}z.
  \end{equation}
  \end{small}
\end{lemma}
\begin{proof}
  For the sake of brevity, we adopt the convention that any equality or inequality
  involving \( f \) is understood to hold $ \mathbb P$-almost surely
  throughout this proof. The argument is divided into four steps.

  \textbf{Step 1.} For any $ \epsilon \in (0,1) $, define
  \[
    A_\epsilon := (\epsilon + A)(1 + \epsilon A)^{-1}.
  \]
  According to \cite[Propositions~3.1.4 and 3.1.9]{Pruss2016},
  the operator \( A_\epsilon \) possesses the following properties:
  \begin{enumerate}
    \item[(a)] \( A_\epsilon \) has a bounded inverse for all \( \epsilon \in (0,1) \);
    \item[(b)] The norm
      \[
        \| z(z - A_\epsilon)^{-1} \|_{\mathcal{L}(L^q(\mathcal{O}))}
      \]
      is uniformly bounded with respect to \( \epsilon \in (0,1) \) and \( z \in \mathbb{C} \setminus \{0\} \) satisfying
      \( \snm{\operatorname{Arg} z} \geqslant \theta_A \);
    \item[(c)] For any \( m \in \mathbb{N}_{>0} \) and for any \( z \in \mathbb{C} \setminus \{0\} \)
      with \( \snm{\operatorname{Arg} z} \geqslant \theta_A \), the operator \( A_\epsilon^{1/2}(z - A_\epsilon)^{-m} \) converges to \( A^{1/2}(z - A)^{-m} \) in \( \mathcal{L}(L^q(\mathcal{O})) \) as \( \epsilon \to {0+} \).
  \end{enumerate}
  For any \( \epsilon \in (0,1) \) and
  \( z \in \mathbb C \setminus \{0\} \) with
  \( \snm{\operatorname{Arg}z} > \theta_A \), we have
  \begin{align*}
    & \norm{z^{1/2}A_\epsilon^{1/2}(z-A_\epsilon)^{-1}}_{\mathcal L(L^q(\mathcal O))} \\
    ={} 
    & \nmB{
      \frac1{2\pi i} \int_{\partial\Sigma_{\theta_A}}
      z^{1/2} \lambda^{1/2} (z-\lambda)^{-1} (\lambda - A_\epsilon)^{-1} \, \mathrm{d}\lambda
    }_{\mathcal L(L^q(\mathcal O))} \\
    \leqslant{}
    & \frac{1}{2\pi} \int_{\partial\Sigma_{\theta_A}}
    \snm{z}^{1/2} \snm{\lambda}^{-1/2} \snm{z-\lambda}^{-1}
    \norm{\lambda(\lambda-A_\epsilon)^{-1}}_{\mathcal L(L^q(\mathcal O))}
    \, \snm{\mathrm{d}\lambda}.
  \end{align*}
  Thus, an elementary calculation using property (b) yields the following additional property:
  \begin{enumerate}
    \item[(d)] 
      For any given $ \theta \in (\theta_A,\pi) $,
      the norm $ \norm{z^{1/2}A_\epsilon^{1/2}(z-A_\epsilon)^{-1}}_{\mathcal L(L^q(\mathcal O))} $
      is uniformly bounded with respect to \( \epsilon \in (0,1) \) and
      $ z \in \mathbb C \setminus\{0\} $ with $ \snm{\operatorname{Arg}z} \geqslant \theta $.
  \end{enumerate}

  \textbf{Step 2.} For any $ \epsilon \in (0,1) $, define $ (Y_{j,\epsilon})_{j\in\mathbb N} $ by
  \[
    \begin{cases}
      Y_{j+1,\epsilon} - Y_{j,\epsilon} +
      \tau A_\epsilon Y_{j+1,\epsilon} = f_j \delta W_j,
      \quad j \in \mathbb N, \\
      Y_{0,\epsilon} = 0.
    \end{cases}
  \]
  By definition we have, for any $ j \geqslant 1 $,
  \begin{align*}
    A^{1/2}Y_j &= \sum_{k=0}^{j-1} A^{1/2}(I + \tau A)^{k-j} f_k \delta W_k, \\
    A_\epsilon^{1/2}Y_{j,\epsilon} &=
    \sum_{k=0}^{j-1} A_{\epsilon}^{1/2}(I + \tau A_\epsilon)^{k-j} f_k \delta W_k.
  \end{align*}
  Hence, using property (c) from Step 1 gives
  \begin{equation}
    \label{eq:Yeps-Y}
    \lim_{\epsilon \to {0+}}
    \norm{A_\epsilon^{1/2} Y_{j,\epsilon} - A^{1/2}Y_j}_{L^q(\mathcal O)}
    = 0, \quad \forall j \geqslant 1.
  \end{equation}


  \textbf{Step 3.} We now proceed to demonstrate the decomposition
\begin{equation}
  \label{eq:Aeps12Y}
  A_\epsilon^{1/2}Y_{j,\epsilon} = I_{j,\epsilon}^{(1)} + I_{j,\epsilon}^{(2)} + I_{j,\epsilon}^{(3)},
  \quad \forall j \geqslant 1, \, \forall \epsilon \in (0,1),
\end{equation}
where
\begin{align*}
  I_{j,\epsilon}^{(1)} &:= \frac{1}{2\pi i} \int_{\Upsilon_1 \cap \{z \in \mathbb{C} \mid \, \operatorname{Im}z < 0\}} A_\epsilon^{1/2}\left(1-e^{-z}+\tau A_\epsilon\right)^{-1} \eta_j(z) \, \mathrm{d}z, \\
  I_{j,\epsilon}^{(2)} &:= \frac{1}{2\pi i} \int_{\Upsilon_1 \cap \{z \in \mathbb{C} \mid \, \operatorname{Im}z > 0\}} A_\epsilon^{1/2}\left(1-e^{-z}+\tau A_\epsilon\right)^{-1} \eta_j(z) \, \mathrm{d}z, \\
  I_{j,\epsilon}^{(3)} &:= \frac{1}{2\pi i} \int_{\Upsilon_2} A_\epsilon^{1/2}\left(1-e^{-z}+\tau A_\epsilon\right)^{-1} \eta_j(z) \, \mathrm{d}z,
\end{align*}
and \(\eta_j(z) := \sum_{k=0}^{j-1} e^{(j-k-1)z} f_k \delta W_k\).
To establish this, we fix any \(\epsilon \in (0,1)\) and \(j \geqslant 1\). Employing the standard discrete Laplace transform method, we obtain
\begin{align*}
  Y_{j,\epsilon} &= \frac{1}{2\pi i} \int_{(1-i\pi, 1+i\pi)} e^{(j-1)z} (1-e^{-z} + \tau A_\epsilon)^{-1} \sum_{k=0}^\infty e^{-kz} f_k \delta W_k \, \mathrm{d}z \\
                 &= \frac{1}{2\pi i} \int_{(1-i\pi, 1+i\pi)} (1-e^{-z} + \tau A_\epsilon)^{-1} \sum_{k=0}^\infty e^{(j-k-1)z} f_k \delta W_k \, \mathrm{d}z.
\end{align*}
By property (b) from Step 1 and the observation that \(|\operatorname{Arg}(e^{-z}-1)| > \pi/2\)
for all \(z \in \mathbb{C}\) with \(\operatorname{Re}z \geqslant 1\),
along with the fact that \( e^{x-i\pi} = e^{x+i\pi} \) for all
\( x \in [1,\infty) \), application of Cauchy's theorem yields the following identity for all $ m \in \mathbb N_{>0} $:
\[
  \int_{(1-i\pi, 1+i\pi)} (1 - e^{-z} + \tau A_{\epsilon})^{-1} e^{-mz} \, \mathrm{d}z = 0.
\]
  This simplifies the expression for \( Y_{j,\epsilon} \) to
  \[
    Y_{j,\epsilon} = \frac{1}{2\pi i} \int_{(1-i\pi, 1+i\pi)}
    \left(1 - e^{-z} + \tau A_{\epsilon}\right)^{-1}
    \eta_j(z) \, \mathrm{d}z.
  \]
  Given that the integrand in the above integral is analytic over
  the closure of the set \( \{z \in \Sigma_{\beta_A} \mid \, \operatorname{Re} z \leqslant 1, \, -\pi \leqslant \operatorname{Im} z \leqslant \pi\} \),
  as certified by \cref{eq:lost} and properties (a) and (b) from Step 1, 
  we can invoke Cauchy's theorem to recast \( Y_{j,\epsilon} \) as
  an integral over \( \Upsilon \):
  \[
    Y_{j,\epsilon} = \frac{1}{2\pi i} \int_{\Upsilon}
    \left(1 - e^{-z} + \tau A_{\epsilon}\right)^{-1} \eta_j(z) \, \mathrm{d}z.
  \]
  Applying $ A_{\epsilon}^{1/2} $ to both sides of the above equality and partitoning
  the integral curve $ \Upsilon $ into three parts
  $ \Upsilon_1 \cap \{z \in \mathbb C \mid \operatorname{Im}z < 0\} $,
  $ \Upsilon_1 \cap \{z \in \mathbb C \mid \operatorname{Im}z > 0\} $,
  and $ \Upsilon_2 $, we immediately arrive at the desired decomposition \cref{eq:Aeps12Y}.

\textbf{Step 4.}
Let \( j \geqslant 1 \) be fixed. For \( I_{j,\epsilon}^{(1)} \), we have
\begin{align*}
    I_{j,\epsilon}^{(1)}
    &= \frac{1}{2\pi i} \int_0^{r_A} \frac{
      e^{i\alpha_A} A_\epsilon^{1/2}(-re^{i\alpha_A} + \tau A_\epsilon)^{-1}
      \eta_j\big( -\log(1+re^{i\alpha_A}) \big)
    }{1 + re^{i\alpha_A}}
    \, \mathrm{d}r  \\
    &= \frac{1}{2\pi i} \int_{\tau/r_A}^\infty
    \frac{
      e^{i\alpha_A} A_\epsilon^{1/2}
      \big( -e^{i\alpha_A} + rA_\epsilon \big)^{-1}
      \eta_j\left(-\log\left( 1+\frac{\tau e^{i\alpha_A}}{r} \right)\right)
    }{r+\tau e^{i\alpha_A}}\, \mathrm{d}r,
\end{align*}
where the first equality follows from the change of variable
\[ r = \left(e^{-z} - 1\right)e^{-i\alpha_A} \quad \text{for } z \in \Upsilon_1 \text{ with } \operatorname{Im} z < 0, \]
and the second equality is obtained by substituting \( r := \tau / r \).
Given the properties (c) and (d) established in Step 1, the condition \( \alpha_A \in (\theta_A, \pi/2) \), and the inequality
\[
  \sup_{r \geqslant \tau/r_A}\left\| \eta_j\left(-\log\left(1+\frac{\tau e^{i\alpha_A}}{r}\right)\right) \right\|_{L^q(\mathcal{O})}
  \leqslant \sum_{k=0}^{j-1} \|f_k\delta W_k\|_{L^q(\mathcal{O})},
\]
the application of Lebesgue's dominated convergence theorem leads to
\[
  \lim_{\epsilon \to 0^+} I_{j,\epsilon}^{(1)} = 
  \frac{1}{2\pi i} \int_{\tau/r_A}^{\infty}
  \frac{e^{i\alpha_A} A^{1/2}(-e^{i\alpha_A}+rA)^{-1}
  \eta_j\left(-\log\left(1+\frac{\tau e^{i\alpha_A}}{r}\right)\right)}
  {r+\tau e^{i\alpha_A}} \, \mathrm{d}r.
\]
Using \eqref{eq:varphi+}, \eqref{eq:I-def}, and \eqref{eq:scrI+}, this limit can be concisely expressed as
\[
  \lim_{\epsilon \to 0^+} I_{j,\epsilon}^{(1)} = 
  \int_{\tau/r_A}^{\infty} \frac{\varphi_{+}(rA)^2}{r+\tau e^{i\alpha_A}}
  \big( \mathscr{I}_{+}(r)f \big)_j \, \mathrm{d}r,
\]
which is also justified by \cite[Proposition~15.1.4]{HytonenWeis2023}.
Similarly, for \( I_{j,\epsilon}^{(2)} \), we deduce
\[
  \lim_{\epsilon \to 0^+} I_{j,\epsilon}^{(2)} = 
\int_{\tau/r_A}^{\infty} \frac{\varphi_{-}(rA)^2}{r+\tau e^{-i\alpha_A}}
\big(\mathscr{I}_{-}(r)f\big)_j \, \mathrm{d}r.
\]
For \( I_{j,\epsilon}^{(3)} \), using properties (c) and (d) from Step 1, the inequality \eqref{eq:lost}, and the inequality
\[
  \sup_{z \in \Upsilon_2} \|\eta_j(z)\|_{L^q(\mathcal{O})} \leqslant \sum_{k=0}^{j-1} \|f_k \delta W_k\|_{L^q(\mathcal{O})},
\]
using Lebesgue's dominated convergence theorem again gives
\[ \lim_{\epsilon \to 0^+} I_{j,\epsilon}^{(3)} = \frac{1}{2\pi i} \int_{\Upsilon_2} A^{1/2} (1 - e^{-z} + \tau A)^{-1} \eta_j(z) \, \mathrm{d}z. \]
From \cref{eq:varphi,eq:I-def}, it follows that
\[ \lim_{\epsilon \to 0^+} I_{j,\epsilon}^{(3)} = \int_{\Upsilon_2} \Psi(z)^2 \left( \mathcal{I}(z)f \right)_j \, \mathrm{d}z. \]
Combining these limits for \( I_{j,\epsilon}^{(1)} \), \( I_{j,\epsilon}^{(2)} \), and \( I_{j,\epsilon}^{(3)} \), together with the decomposition \eqref{eq:Aeps12Y}, we obtain
\begin{align*}
\lim_{\epsilon \to 0^+} A_\epsilon^{1/2} Y_{j,\epsilon}
& = \int_{\tau/r_A}^{\infty} \left(
\frac{\varphi_{+}(rA)^2}{r+\tau e^{i\alpha_A}} \big( \mathscr{I}_{+}(r)f \big)_j +
\frac{\varphi_{-}(rA)^2}{r+\tau e^{-i\alpha_A}} \big( \mathscr{I}_{-}(r)f \big)_j
\right) \, \mathrm{d}r \\
& \quad + \int_{\Upsilon_2} \Psi(z)^2 \big( \mathcal{I}(z)f \big)_j \, \mathrm{d}z.
\end{align*}
Together with \eqref{eq:Yeps-Y}, this yields
\begin{align*}
  A^{1/2} Y_j 
  &= \int_{\tau/r_A}^{\infty} \left(
    \frac{\varphi_{+}(rA)^2}{r+\tau e^{i\alpha_A}} \big( \mathscr{I}_{+}(r)f \big)_j +
    \frac{\varphi_{-}(rA)^2}{r+\tau e^{-i\alpha_A}} \big( \mathscr{I}_{-}(r)f \big)_j
  \right) \, \mathrm{d}r \\
  & \quad {}+ \int_{\Upsilon_2}
  \Psi(z)^2 \big( \mathcal{I}(z)f \big)_j \, \mathrm{d}z.
\end{align*}
Since the above equality holds for all \( j \geqslant 1 \) and trivially for \( j = 0 \),
the desired equality \eqref{eq:A12Y} is established. This completes the proof.
\end{proof}

Finally, we conclude the proof of Theorem \ref{thm:space-regu} with the following argument.
For the sake of brevity, we adopt the following notation:
\begin{align*}
  X_0 &:= \ell^p(L^p(\Omega; L^q(\mathcal O; H))), \\
  X_1 &:= \ell^p(L^p(\Omega; L^q(\mathcal O))), \\
  X_2 &:= \ell^{p'}(L^{p'}(\Omega; L^{q'}(\mathcal O))), \\
  X_3 &:= \ell^p\left(L^p\left(\Omega; L^q\left(\mathcal O; L^2\left(\mathbb R_{+}, \frac{\mathrm{d}r}{r}\right)\right)\right)\right), \\
  X_4 &:= \ell^{p'}\left(L^{p'}\left(\Omega; L^{q'}\left(\mathcal O; L^2\left(\mathbb R_{+}, \frac{\mathrm{d}r}{r}\right)\right)\right)\right).
\end{align*}
The duality pairing between $X_1$ and $X_2$ is denoted by $\langle \cdot, \cdot \rangle$.
Fix any $ g \in X_2 $. We have
\begin{align*}  
  \dual{A^{1/2}Y,g} 
  &= \int_{\tau/r_A}^\infty \dualB{
    \frac{r}{r+\tau e^{i\alpha_A}} \varphi_{+}(rA)^2 \mathscr I_{+}(r) f +
    \frac{r}{r+\tau e^{-i\alpha_A}} \varphi_{-}(rA)^2 \mathscr I_{-}(r), \, g
  } \frac{\mathrm{d}r}r \\ 
  & \qquad {} + \int_{\Upsilon_2}
  \dualB{\Psi(z)^2 \mathcal I(z) f, \, g} \, \mathrm{d}z \\
  &=
  \int_{\tau/r_A}^\infty \dualB{
    \frac{r}{r+\tau e^{i\alpha_A}} \varphi_{+}(rA) \mathscr I_{+}(r) f, \, \varphi_{+}^*(rA^*)g
  } \frac{\mathrm{d}r}r \\ 
      & \qquad {} + \int_{\tau/r_A}^\infty \dualB{
        \frac{r}{r+\tau e^{-i\alpha_A}} \varphi_{-}(rA) \mathscr I_{-}(r) f, \, \varphi_{-}^*(rA^*) g
      } \frac{\mathrm{d}r}r \\ 
      & \qquad {} + \int_{\Upsilon_2} 
      \dualB{\Psi(z) \mathcal I(z) f, \, \Psi^*(z)g} \, \mathrm{d}z.
\end{align*}
where the first equality follows from the equality \cref{eq:A12Y},
and the second equality is a consequence of the fact that \( \varphi_{\pm}^*(rA^*) \) and \( \Psi^*(z) \) are the adjoint operators
of \( \varphi_{\pm}(rA) \) and \( \Psi(z) \), respectively.
Using the commutative properties between
$ \varphi_{\pm}(rA) $ and $ \mathscr I_{\pm}(r) $,
as well as the commutativity between $ \Psi(z) $ and $ \mathcal I(z) $,
we subsequently deduce that
\begin{align*}  
  \dual{A^{1/2}Y,g} 
  ={} &
  \int_{\tau/r_A}^\infty \dualB{
    \frac{r}{r+\tau e^{i\alpha_A}} \mathscr I_{+}(r) \varphi_{+}(rA) f, \, \varphi_{+}^*(rA^*)g
  } \frac{\mathrm{d}r}r \\ 
      & \qquad {} + \int_{\tau/r_A}^\infty \dualB{
        \frac{r}{r+\tau e^{-i\alpha_A}} \mathscr I_{-}(r) \varphi_{-}(rA) f, \, \varphi_{-}^*(rA^*) g
      } \frac{\mathrm{d}r}r \\ 
      & \qquad {} + \int_{\Upsilon_2}
      \dualB{\mathcal I(z) \Psi(z) f, \, \Psi^*(z)g} \, \mathrm{d}z.
\end{align*}
Extending $ \mathcal I_{+} $ and $ \mathcal I_{-} $ to $ [0,\tau/r_A) $ by zero,
and applying H\"older's inequality along with the uniform boundedness of
$ \frac{r}{r+\tau e^{\pm i\alpha_A}} $ with respect to $ r \in [0,\infty) $,
we further obtain
\begin{align*}
  \snmB{\dual{A^{1/2}Y,g}}
  &\leqslant
  \nmB{
    \mathscr I_{+}(r) \varphi_{+}(rA)f
  }_{X_3} \norm{\varphi_{+}^*(rA^*)g}_{X_4} + \nmB{
    \mathscr I_{-}(r) \varphi_{-}(rA)f
  }_{X_3} \norm{\varphi_{-}^*(rA^*)g}_{X_4} \\
  & \quad{} + \norm{\mathcal I(\cdot)\Psi(\cdot)f}_{
    \ell^p(L^p(\Omega;L^q(\mathcal O;L^2(\Upsilon_2,\snm{\mathrm{d}z}))))
  } \norm{\Psi^*(\cdot)g}_{
    \ell^{p'}(L^{p'}(\Omega;L^{q'}(\mathcal O;L^2(\Upsilon_2,\snm{\mathrm{d}z}))))
  } .
\end{align*}
Then, utilizing the $\gamma$-Fubini isomorphism (\cite[Theorem~9.4.8]{HytonenWeis2017}) yields
\begin{align*}
  \snmB{\dual{A^{1/2}Y,g}} 
  &\leqslant
  c\nmB{
    \mathscr I_{+}(r) \varphi_{+}(rA)f
  }_{
    \gamma(L^2(\mathbb R_{+},\frac{\mathrm{d}r}r), X_1)
  } \norm{\varphi_{+}^*(rA^*)g}_{
    \gamma(L^2(\mathbb R_{+},\frac{\mathrm{d}r}r), X_2)
  } \\
  & \qquad {} + c\nmB{
    \mathscr I_{-}(r) \varphi_{-}(rA)f
  }_{
    \gamma(L^2(\mathbb R_{+},\frac{\mathrm{d}r}r),  X_1)
  } \norm{\varphi_{-}^*(rA^*)g}_{
    \gamma(L^2(\mathbb R_{+},\frac{\mathrm{d}r}r), X_2)
  } \\
  & \qquad\ {} + c\norm{\mathcal I(\cdot)\Psi(\cdot)f}_{
    \gamma(L^2(\Upsilon_2,\snm{\mathrm{d}z}),X_1)
  } \norm{\Psi^*(\cdot)g}_{
    \gamma(L^2(\Upsilon_2,\snm{\mathrm{d}z}),X_2)
  }.
\end{align*}
According to Lemma~\ref{lem:I-R-bounded}, the $\mathcal{R}$-boundedness of the set $\{\mathcal{I}(z) \mid z \in \Upsilon\}$ is independent of $\tau$. Since $\{\mathscr{I}_{+}(r) \mid r \in [0,\infty)\}$ and $\{\mathscr{I}_{-}(r) \mid r \in [0,\infty)\}$ are subsets of $\{\mathcal{I}(z) \mid z \in \Upsilon\}$, their $\mathcal{R}$-boundedness is also independent of $\tau$.
Therefore, applying the $\gamma$-Multiplier theorem (Theorem~9.5.1 in \cite{HytonenWeis2017}) and noting that $\mathcal{R}$-boundedness implies $\gamma$-boundedness (Theorem~8.1.3 in \cite{HytonenWeis2017}), we derive from the previous inequality that
\begin{align*}
  \left| \langle A^{1/2}Y, g \rangle \right|
  &\leqslant
  c \left\| \varphi_{+}(rA)f \right\|_{\gamma(L^2(\mathbb{R}_+,\frac{\mathrm{d}r}{r}), X_0)} \left\| \varphi_{+}^*(rA^*)g \right\|_{\gamma(L^2(\mathbb{R}_+,\frac{\mathrm{d}r}{r}), X_2)} \\
  &\quad +
  c \left\| \varphi_{-}(rA)f \right\|_{\gamma(L^2(\mathbb{R}_+,\frac{\mathrm{d}r}{r}), X_0)} \left\| \varphi_{-}^*(rA^*)g \right\|_{\gamma(L^2(\mathbb{R}_+,\frac{\mathrm{d}r}{r}), X_2)} \\
  &\quad +
  c \left\| \Psi(\cdot)f \right\|_{\gamma(L^2(\Upsilon_2,|\mathrm{d}z|), X_0)} \left\| \Psi^*(\cdot)g \right\|_{\gamma(L^2(\Upsilon_2,|\mathrm{d}z|), X_2)}.
\end{align*}
This inequality, combined with the square function estimates in \cref{lem:varphi+-,lem:varphi+-*,lem:varphi,lem:varphi*}, leads to
\begin{align*}
  \snmB{ \dual{A^{1/2}Y,g} } 
  \leqslant c  \norm{f}_{X_0} \norm{g}_{X_2}.
\end{align*}
Since \( X_2 \) is the dual of \( X_1 \) and \( g \) is arbitrarily chosen
from \( X_2 \), the desired estimate \eqref{eq:space-regu} is established,
thereby completing the proof of Theorem \ref{thm:space-regu}.

\begin{remark}
  \label{rem:DSMLp}
  Let \( p \in (2,\infty) \) and \( q \in [2,\infty) \). Denote by \( L_\mathbb{F}^p(\Omega \times \mathbb{R}_+; L^q(\mathcal{O}; H)) \) the space of all \( \mathbb{F} \)-adapted \( L^q(\mathcal{O}; H) \)-valued processes that belong to \( L^p(\Omega \times \mathbb{R}_+; L^q(\mathcal{O}; H)) \).
  We redefine \( \{\mathcal{I}(z) \mid z \in \Upsilon\} \) as a family
  of operators acting from \( L_\mathbb{F}^p(\Omega \times \mathbb{R}_+; L^q(\mathcal{O}; H)) \) to \( \ell^p(L^p(\Omega; L^q(\mathcal{O}))) \) as follows: for any \( z \in \Upsilon \) and \( g \in L_\mathbb{F}^p(\Omega \times \mathbb{R}_+; L^q(\mathcal{O}; H)) \), \( \mathcal{I}(z)g \) is given by
  \begin{numcases}{}
      (\mathcal{I}(z)g)_0 = 0, \notag \\
      (\mathcal{I}(z)g)_j = \frac1{\sqrt\tau}\sum_{k=0}^{j-1} (e^{-z} - 1)^{1/2} e^{(j-k-1)z} \int_{k\tau}^{k\tau + \tau} g(t) \, \mathrm{d}W(t), \quad j \geqslant 1. \notag
    \end{numcases}
  By slightly modifying the proof of Lemma \ref{lem:I-R-bounded}, it can be shown that the
  \( \mathcal{R} \)-boundedness of \( \{\mathcal{I}(z) \mid z \in \Upsilon\} \) is bounded
  by \( c\tau^{-1/p} \), where \( c \) is a constant independent of \( \tau \).
  Consequently, following the proof of Theorem \ref{thm:space-regu}, we deduce the following form
  of discrete stochastic maximal $ L^p $-regularity estimate:
  \[
    \left[ \mathbb{E} \sum_{j=1}^\infty \tau \norm{A^{1/2} Y_j}_{L^q(\mathcal{O})}^p \right]^{1/p}
    \leqslant c \norm{g}_{L^p(\Omega \times \mathbb{R}_+; L^q(\mathcal{O}; H))}
  \]
  for all \( g \in L_\mathbb{F}^p(\Omega \times \mathbb{R}_+; L^q(\mathcal{O}; H)) \), where \( (Y_j)_{j=1}^\infty \) is defined by the Euler scheme:
  \[
    \begin{cases}
      Y_{j+1} - Y_j + \tau A Y_{j+1} = \int_{j\tau}^{j\tau + \tau} g(t) \, \mathrm{d}W(t), \quad j \geqslant 1, \\
      Y_0 = 0.
    \end{cases}
  \]
\end{remark}

\section{Convergence estimate}
\label{sec:convergence}
Let $ J $ be a positive integer and define the time step by $ \tau := T/J $.
For any $ p, q \in (1,\infty) $, let $ L_{\mathbb F,\tau}^p(\Omega\times(0,T);L^q(\mathcal O;H)) $
denote the space of all processes $ f: \Omega\times[0,T]\to L^q(\mathcal O;H) $
that are piecewise constant on each time interval $ [{j\tau},{j\tau+\tau}) $ and satisfy
\[
f({j\tau}) \in L^p(\Omega, \mathcal{F}_{{j\tau}}, \mathbb{P}; L^q(\mathcal{O}; H))
\]
for all $ 0 \leqslant j \leqslant J $.

The main result of this section is the following convergence estimate.
\begin{theorem}
  \label{thm:conv}
  Let $ p,q \in [2,\infty) $. Suppose that $ A $ is a densely defined sectorial operator on $ L^q(\mathcal O) $
  with a bounded inverse. Assume further that the family
  \(
  \left\{ z(z-A)^{-1} \mid z \in \mathbb C \setminus \overline{\Sigma_{\theta_A}} \right\}
  \)
  is $ \mathcal R $-bounded in $ \mathcal L(L^q(\mathcal O)) $,
  where $ \theta_A \in (0,\pi/2) $. Let $ y $ be the mild solution to
  \begin{equation} 
    \label{eq:y}
    \begin{cases}
      \mathrm{d}y(t) + Ay(t) \, \mathrm{d}t =
      f(t) \, \mathrm{d}W(t), \quad 0 \leqslant t \leqslant T, \\
      y(0) = 0,
    \end{cases}
  \end{equation}
  with \( f \in L_{\mathbb F,\tau}^p(\Omega\times(0,T);L^q(\mathcal O;H)) \).
  Define the sequence $ (Y_j)_{j=0}^J $ by
  \begin{subequations}
    \label{eq:Y-J-def}
    \begin{numcases}{}
      Y_{j+1} - Y_j + \tau A Y_{j+1} = \int_{{j\tau}}^{{j\tau+\tau}} f(t) \, \mathrm{d}W(t),
      \quad 0 \leqslant j < J, \\
      Y_0 = 0.
    \end{numcases}
  \end{subequations}
  Then, the following error estimate holds:
  \begin{equation} 
    \label{eq:conv}
    \biggl[
      \mathbb E 
      \sum_{j=0}^{J-1} \int_{{j\tau}}^{{j\tau+\tau}} \norm{y(t)-Y_j}_{ L^q(\mathcal O) }^p \, \mathrm{d}t
    \biggr]^{1/p} \leqslant
    c \tau^{1/2} \norm{f}_{L^p(\Omega\times(0,T);L^q(\mathcal O;H))},
  \end{equation}
  where $ c $ is a constant independent of the time step $ \tau $.
\end{theorem}

\begin{remark}
  Suppose $ p \in (2, \infty) $ and $ q \in [2, \infty) $. Assume that the operator $ A $ satisfies the
  hypotheses of Theorem \ref{thm:space-regu}, with the added premise that $ A $ possesses a bounded inverse.
  Let $ y $ be the mild solution of equation \ref{eq:y}
  with $ f \in L_{\mathbb{F}}^p(\Omega \times (0,T); L^q(\mathcal{O}; H)) $,
  where $ L_{\mathbb{F}}^p(\Omega \times (0,T); L^q(\mathcal{O}; H)) $ is defined as in \cref{rem:DSMLp}.
  Let $ (Y_j)_{j=0}^J $ be the solution of \cref{eq:Y-J-def}.
  We can establish the following error estimate (the detailed proof is left to the interested reader):
  \[
    \left[
      \mathbb{E} \sum_{j=0}^{J-1} \int_{{j\tau}}^{{j\tau+\tau}} \norm{y(t) - Y_j}_{L^q(\mathcal{O})}^p \, \mathrm{d}t
    \right]^{1/p} \leqslant c \tau^{\frac12 - \frac1p - \epsilon} \norm{f}_{L^p(\Omega \times (0,T); L^q(\mathcal{O}; H))},
  \]
  where \( \epsilon > 0 \) can be chosen arbitrarily small.
  Notably, the temporal convergence rate can be enhanced to \( \frac{1}{2} \) when the process \( f \) is piecewise
  constant in time, as established in Theorem \ref{thm:conv}. Such error estimates are particularly valuable for
  the numerical analysis of stochastic optimal control problems involving stochastic evolution equations.
\end{remark}

The purpose of the rest of this section is to prove the above theorem.
For ease of reference, we will consistently assume throughout the remainder of this section that
\( p, q \in [2, \infty) \), and that the operator \( A \) satisfies the conditions in Theorem \ref{thm:conv}.
As discussed in Subsection \ref{ssec:space-regu}, let \( A^* \) denote the dual operator of \( A \),
which is a sectorial operator on \( L^{q'}(\mathcal{O}) \). We denote by \( D(A^*) \) the domain of
\( A^* \), endowed with the standard graph norm defined as
\[
  \norm{v}_{D(A^*)} := \norm{A^*v}_{L^{q'}(\mathcal{O})}, \quad \forall v \in D(A^*).
\]
Furthermore, according to \cite[Proposition~8.4.1]{HytonenWeis2017}, the family of operators
\[
  \left\{ z(z - A^*)^{-1} \mid z \in \mathbb{C} \setminus \overline{\Sigma_{\theta_A}} \right\}
\]
is \( \mathcal{R} \)-bounded in \( \mathcal{L}(L^{q'}(\mathcal{O})) \). Additionally,
we introduce the vector-valued Sobolev space
\[
  {}^0H^{1,p'}(0,T;L^{q'}(\mathcal{O})) := \left\{ v : [0,T] \to L^{q'}(\mathcal{O}) \mid v' \in L^{p'}(0,T;L^{q'}(\mathcal{O})), \, v(T) = 0 \right\},
\]
where \( v' \) denotes the weak derivative of $ v $ with respect to the time variable.

\medskip\noindent\textbf{Proof of \cref{thm:conv}.}
We split the proof into the following two steps.

\textbf{Step 1.}
  Let \( v \) be an arbitrary but fixed element of
  $ {}^0H^{1,p'}(0,T; L^{q'}(\mathcal{O})) \cap L^{p'}(0,T; D(A^*)) $.
  Using standard techniques (see, e.g., \cite[Lemma~5.5]{Prato2014}), we have, almost surely,
  \[
    \int_0^T \dual{ y(t), \, -v'(t) + A^*v(t)} \, \mathrm{d}t = \int_0^T \langle v(t), \, f(t) \rangle \, \mathrm{d}W(t),
  \]
 where the duality pairing on the left-hand side is between $ L^q(\mathcal O) $ and
 $ L^{q'}(\mathcal O) $, and the pairing on the right-hand side denotes an
 $ \mathbb F $-adapted $ \gamma(H,\mathbb R) $-valued process, as detailed in Remark \ref{rem:love}.
 Furthermore, a direct computation yields, almost surely,
  \begin{align*}
    \int_0^T \langle v(t), \, f(t) \rangle \, \mathrm{d}W(t)
    &= \sum_{j=0}^{J-2} \biggl\langle \int_{{j\tau}}^{{j\tau+\tau}} f({j\tau}) \, \mathrm{d}W(t), \, v({j\tau+\tau}) \biggr\rangle \\
    &\quad - \sum_{j=0}^{J-1} \int_{{j\tau}}^{{j\tau+\tau}} \biggl\langle \int_{{j\tau}}^t f({j\tau}) \, \mathrm{d}W(s), \, v'(t) \biggr\rangle \, \mathrm{d}t.
  \end{align*}
  Hence, almost surely,
  \begin{align*}
  & \int_0^T \dual{y(t), \, -v'(t) + A^*v(t)} \, \mathrm{d}t \\
  ={}
  &\sum_{j=0}^{J-2} \biggl\langle \int_{{j\tau}}^{{j\tau+\tau}} f({j\tau}) \, \mathrm{d}W(t), \, v({j\tau+\tau}) \biggr\rangle 
  - \sum_{j=0}^{J-1} \int_{{j\tau}}^{{j\tau+\tau}} \biggl\langle \int_{{j\tau}}^t f({j\tau}) \, \mathrm{d}W(s), \, v'(t) \biggr\rangle \, \mathrm{d}t.
  \end{align*}
  For any measurable set \( C \in \mathcal{F} \), multiplying both sides of the above equality by the indicator
  function \( \mathbbm{1}_C \) for \( C \) and taking expectations, we deduce that
  \begin{equation}
    \label{eq:y-z}
    \begin{aligned}
      & \mathbb{E} \int_0^T \dual{y(t), \, -z'(t) + A^*z(t)} \, \mathrm{d}t \\
      ={}
      & \mathbb{E} \Biggl\{
        \sum_{j=0}^{J-2} \biggl\langle \int_{{j\tau}}^{{j\tau+\tau}} f({j\tau}) \, \mathrm{d}W(t), \, z({j\tau+\tau}) \biggr\rangle 
        - \sum_{j=0}^{J-1} \int_{{j\tau}}^{{j\tau+\tau}} \biggl\langle \int_{{j\tau}}^t f({j\tau}) \, \mathrm{d}W(t), \, z'(t) \biggr\rangle \, \mathrm{d}t
      \Biggr\}
    \end{aligned}
  \end{equation}
  holds for \( z = v \mathbbm{1}_C \). Since both sides of \cref{eq:y-z} act as bounded linear functionals on
  \[
    L^{p'}(\Omega; {}^0H^{1,p'}(0,T; L^{q'}(\mathcal{O})) \cap L^{p'}(0,T; D(A^*)))
  \]
  with respect to $ z $, as can be readily verified by \cref{lem:integral}, and
  considering the density of the linear span
  \[
    \text{span}\left\{
      v \mathbbm{1}_C \mid v \in {}^0H^{1,p'}(0,T; L^{q'}(\mathcal{O})) \cap L^{p'}(0,T; D(A^*)), \, C \in \mathcal{F}
    \right\}
  \]
  within this space, we can apply a density argument to conclude that \cref{eq:y-z} holds
  for all 
  \[
    z \in L^{p'}(\Omega; {}^0H^{1,p'}(0,T; L^{q'}(\mathcal{O})) \cap L^{p'}(0,T; D(A^*))).
  \]

  \textbf{Step 2.}
  Let \( g \in L^{p'}(\Omega \times (0,T); L^{q'}(\mathcal{O})) \) be given arbitrarily.
  According to Theorem 4.2 in \cite{Weis2001}, there exists a process \( z \) such that \( z \) solves the
  backward evolution equation almost surely:
  \[
    \begin{cases}
      -z'(t) + A^*z(t) = g(t), & 0 \leqslant t \leqslant T, \\
      z(T) = 0,
    \end{cases}
  \]
  and satisfies the regularity estimate
  \begin{equation} 
    \label{eq:z-regu}
    \norm{z'}_{L^{p'}(\Omega \times (0,T); L^{q'}(\mathcal{O}))} 
    + \norm{A^*z}_{L^{p'}(\Omega \times (0,T); L^{q'}(\mathcal{O}))} 
    \leqslant c \norm{g}_{L^{p'}(\Omega \times (0,T); L^{q'}(\mathcal{O}))}.
  \end{equation}
  Define the sequence \( (Z_j)_{j=0}^J \) almost surely by
  \begin{subequations}
    \label{eq:Z}
    \begin{numcases}{}
      Z_j - Z_{j+1} + \tau A^*Z_j = \int_{{j\tau}}^{{j\tau+\tau}} g(t) \, \mathrm{d}t, \quad  0 \leqslant j < J, \\
      Z_J = 0.
    \end{numcases}
  \end{subequations}
  By adopting an approach analogous to that presented in \cite[Theorem III]{Kemmochi2018},
  we can establish the following well-known inequality:
  \begin{equation}
    \label{eq:z-Z-omega}
    \biggl[
      \mathbb E\sum_{j=0}^{J-1} \tau \norm{z({j\tau}) - Z_j}_{L^{q'}(\mathcal{O})}^{p'}
    \biggr]^{1/p'} \leqslant c \tau \norm{g}_{L^{p'}(\Omega \times (0,T); L^{q'}(\mathcal{O}))}.
  \end{equation}
Next, using \cref{eq:Y-J-def} and \cref{eq:Z}, a straightforward computation gives that
\begin{align*} 
  \mathbb E\sum_{j=0}^{J-1} \int_{{j\tau}}^{{j\tau+\tau}} \dualb{ Y_j, \, g(t) } \, \mathrm{d}t
  = \mathbb E\sum_{j=0}^{J-2} \biggl\langle
    \int_{{j\tau}}^{{j\tau+\tau}} f({j\tau}) \, \mathrm{d}W(t), \, Z_{j+1}
  \biggr\rangle.
\end{align*}
Furthermore, the identity \cref{eq:y-z} implies that
\begin{align*}
  & \mathbb E \int_0^T \langle y(t), \, g(t) \rangle \, \mathrm{d}t \\
  ={}
  & \mathbb E \Bigg\{
    \sum_{j=0}^{J-2} \biggl\langle \int_{{j\tau}}^{{j\tau+\tau}}f({j\tau})\,\mathrm{d}W(t), \, z({j\tau+\tau}) \biggr\rangle 
    - \sum_{j=0}^{J-1} \int_{{j\tau}}^{{j\tau+\tau}} 
    \biggl\langle \int_{{j\tau}}^t f({j\tau}) \, \mathrm{d}W(t), \, z'(t) \biggr\rangle  \, \mathrm{d}t
  \Bigg\}.
\end{align*}
Combining these equalities, we derive the following identity:
\begin{align*}
  \mathbb{E}\sum_{j=0}^{J-1} \int_{{j\tau}}^{{j\tau+\tau}} \langle y(t) - Y_j, g(t) \rangle \, \mathrm{d}t 
  &= \mathbb{E}\sum_{j=0}^{J-2} \biggl\langle \int_{{j\tau}}^{{j\tau+\tau}} f({j\tau}) \, \mathrm{d}W(t),
    \, z({j\tau+\tau}) - Z_{j+1}
  \biggr\rangle \\
  & \qquad {} - \mathbb{E}\sum_{j=0}^{J-1} \int_{{j\tau}}^{{j\tau+\tau}}
  \biggl\langle \int_{{j\tau}}^t f({j\tau})\,\mathrm{d}W(s), \, z'(t) \biggr\rangle \, \mathrm{d}t \\
  &=: I_1 + I_2.
\end{align*}
  For \( I_1 \), we have
  \begin{align*}
    I_1
    & \leqslant \left(
      \sum_{j=0}^{J-2} \left\| \int_{j\tau}^{j\tau+\tau} f(j\tau) \, \mathrm{d}W(t) \right\|_{L^p(\Omega; L^q(\mathcal O))}^p
    \right)^{\frac{1}{p}} \left(
      \sum_{j=0}^{J-2} \left\| z(j\tau+\tau) - Z_{j+1} \right\|_{
        L^{p'}(\Omega; L^{q'}(\mathcal O))
      }^{p'}
    \right)^{\frac{1}{p'}} \\
   & \leqslant c \tau^{\frac{1}{p}} \left(
     \sum_{j=0}^{J-2} \left\| \int_{j\tau}^{j\tau+\tau}f(j\tau) \, \mathrm{d}W(t) \right\|_{L^p(\Omega; L^q(\mathcal O))}^p
   \right)^{\frac{1}{p}} \left\| g \right\|_{L^{p'}(\Omega\times(0,T); L^{q'}(\mathcal O))}
   \quad \text{(by \cref{eq:z-Z-omega})} \\
   & \leqslant c \tau^{\frac{1}{p}} \left(
     \sum_{j=0}^{J-2} \tau^{\frac{p}2} \left\| f(j\tau) \right\|_{L^p(\Omega; L^q(\mathcal O;H))}^p
   \right)^{\frac{1}{p}} \left\| g \right\|_{L^{p'}(\Omega\times(0,T); L^{q'}(\mathcal O))}
   \quad \text{(by \cref{lem:integral})} \\
   &= c\tau^{\frac12} \norm{f}_{L^p(\Omega\times(0,T);L^q(\mathcal O;H))}
   \norm{g}_{L^{p'}(\Omega\times(0,T);L^{q'}(\mathcal O))}.
\end{align*}
Similarly, for \( I_2 \), using Hölder's inequality, the regularity estimate \cref{eq:z-regu},
and \cref{lem:integral}, we deduce that
\[
  I_2 \leqslant c \tau^{1/2}
  \left\| f \right\|_{L^p(\Omega\times(0,T); L^q(\mathcal{O}; H))}
  \left\| g \right\|_{L^{p'}(\Omega\times(0,T); L^{q'}(\mathcal{O}))}.
\]
By combining these bounds, we obtain
\begin{align*}
  \mathbb{E} \sum_{j=0}^{J-1} \int_{{j\tau}}^{{j\tau+\tau}} \langle y(t) - Y_j, \, g \rangle \, \mathrm{d}t 
  \leqslant c \tau^{1/2} \left\| f \right\|_{L^p(\Omega\times(0,T); L^q(\mathcal{O};H))}
  \left\| g \right\|_{L^{p'}(\Omega\times(0,T); L^{q'}(\mathcal{O}))}.
\end{align*}
Given that \( g \in L^{p'}(\Omega\times(0,T); L^{q'}(\mathcal{O})) \) is arbitrary,
we can invoke the principle of duality to directly achieve the desired error estimate \cref{eq:conv}.
This concludes the proof of \cref{thm:conv}.

\hfill\ensuremath{\blacksquare}

\begin{remark}
  \label{rem:love}
  In Step 1 of the preceding proof, the stochastic integral \( \int_0^T \langle v(t), f(t) \rangle \, \mathrm{d}W(t) \) is
  interpreted as follows. By virtue of the \( \gamma \)-Fubini isomorphism
  (cf. \cite[Theorem~9.4.8]{HytonenWeis2017}), it is established that \( f \) is an \( \mathbb{F} \)-adapted
  \( \gamma(H, L^q(\mathcal{O})) \)-valued process. Furthermore, for each \( t \in [0,T] \), \( v(t) \) can
  be viewed as a bounded linear functional from \( L^q(\mathcal{O}) \) to \( \mathbb{R} \). Consequently,
  by applying the ideal property of \( \gamma \)-radonifying operators (see \cite[Theorem~9.1.10]{HytonenWeis2017}),
  \( \langle v(\cdot), f(\cdot) \rangle \) is identified as an \( \mathbb{F} \)-adapted \( \gamma(H, \mathbb{R}) \)-valued process. This identification comes with the estimate
  \[
    \norm{\langle v(t), f(t) \rangle}_{\gamma(H, \mathbb{R})} \leqslant \norm{v(t)}_{L^{q'}(\mathcal{O})} \cdot \norm{f(t)}_{\gamma(H, L^q(\mathcal{O}))}
    \leqslant c \norm{v(t)}_{L^{q'}(\mathcal O)}
    \cdot\norm{f(t)}_{L^q(\mathcal O;H)}, \quad t \in [0,T].
  \]
  Therefore, the stochastic integral \( \int_0^T \langle v(t), f(t) \rangle \, \mathrm{d}W(t) \) is understood
  as the stochastic integral of an \( \mathbb{F} \)-adapted \( \gamma(H, \mathbb{R}) \)-valued process.
\end{remark}

\begin{remark}
To provide further insight into the derivation of \cref{eq:z-Z-omega}, we start by noting that, almost surely,
\begin{align*}
\big(z(j\tau) - Z_j\big) - \big(z(j\tau+\tau) - Z_{j+1}\big) + \tau A^* \big(z(j\tau) - Z_j\big)
= A^*\int_{j\tau}^{j\tau + \tau} z(j\tau) - z(t) \, \mathrm{d}t.
\end{align*}
Utilizing the deterministic discrete maximal $L^p$-regularity estimate (see \cite[Theorem~3.2]{Kemmochi2016}),
we obtain
\begin{align*}
\left[
\mathbb{E}\sum_{j=0}^{J-1} \tau \norm{z(j\tau) - Z_j}_{L^{q'}(\mathcal{O})}^{p'}
\right]^{1/p'}
\leqslant c \left[
\mathbb{E} \sum_{j=0}^{J-1} \int_{j\tau}^{j\tau + \tau}
\norm{z(t) - z(j\tau)}_{L^{q'}(\mathcal{O})}^{p'} \, \mathrm{d}t
\right]^{1/p'}.
\end{align*}
The desired inequality \cref{eq:z-Z-omega} then follows from the standard estimate:
\[
\left[
\mathbb{E} \sum_{j=0}^{J-1} \int_{j\tau}^{j\tau + \tau}
\norm{z(t) - z(j\tau)}_{L^{q'}(\mathcal{O})}^{p'} \, \mathrm{d}t
\right]^{1/p'}
\leqslant c \tau \norm{z'}_{L^{p'}(\Omega \times (0,T); L^{q'}(\mathcal{O}))}.
\]
\end{remark}

\appendix
\section{Some technical estimates}
\label{sec:some_proofs}
In this section, we employ the notation established in \cref{sec:stability}.
Let $\mu$ denote the Lebesgue measure on the domain $\mathcal{O}$.
We also recall that $c$ represents a generic positive constant,
which is independent of the time step $\tau$, though its value may vary from one instance to another.

\begin{lemma}
  \label{lem:calI}
  Suppose that \( p, q \in [2, \infty) \). Consider the curve \( \Upsilon \) constructed in Subsection
  \ref{ssec:space-regu}. Let \( \mathcal{I}(z) \) be defined for \( z \in \Upsilon \) as in \cref{eq:I-def}.
  We assert that \( \mathcal{I}(z) \) is uniformly bounded in the norm of the space
  \[
  \mathcal{L}\left( \ell_{\mathbb{F}}^p(L^p(\Omega; L^q(\mathcal{O}; H))), \, \ell^p(L^p(\Omega; L^q(\mathcal{O}))) \right)
  \]
  with respect to both \( z \in \Upsilon \) and the time step \( \tau \).
\end{lemma}
\begin{proof}
  Fix any $ z \in \Upsilon $ and let $ \xi = e^{-z} - 1 $.
  Using \cref{lem:integral}, Minkowski's inequality and H\"older's inequality,
  we obtain, for any $ g \in \ell_\mathbb F^p(L^p(\Omega;L^q(\mathcal O;H))) $,
  \begin{align*}
    \norm{\mathcal I(z)g}_{\ell^p(L^p(\Omega;L^q(\mathcal O)))}^p 
    ={} &
    \sum_{j=1}^\infty \nmB{
      \sum_{k=0}^{j-1} \xi^{1/2} (1 + \xi)^{k+1-j}
      g_k \delta W_k/\sqrt\tau
    }_{L^p(\Omega;L^q(\mathcal O))}^p \\
    \leqslant{} &
    c\sum_{j=1}^\infty \mathbb E \biggl[
      \int_\mathcal O \Big(
        \sum_{k=0}^{j-1} \norm{\xi^{1/2}(1+\xi)^{k+1-j} g_k}_H^2
      \Big)^{q/2} \, \mathrm{d}\mu
    \biggr]^{p/q} \\
    \leqslant{} &
    c \sum_{j=1}^\infty \mathbb E \biggl[
      \sum_{k=0}^{j-1} \Big(
        \int_\mathcal O
        \norm{\xi^{1/2}(1+\xi)^{k+1-j}g_k}_H^q \, \mathrm{d}\mu
      \Big)^{2/q}
    \biggr]^{p/2} \\
    \leqslant{} &
    c \sum_{j=1}^\infty \mathbb E \Bigg\{\left[
        \sum_{k=0}^{j-1}
        \snmb{\xi^{1/2} (1+\xi)^{k+1-j}}^2
        \norm{g_k}_{L^q(\mathcal O;H)}^p
      \right]
      \big( I(j,\xi) \big)^{p/2-1}
    \Bigg\},
  \end{align*}
  where
  \begin{align*}
    I(j,\xi) := \sum_{k=0}^{j-1} \snmb{
      \xi^{1/2} (1+\xi)^{k+1-j}
    }^2 = \sum_{k=0}^{j-1} \snmb{
      \xi^{1/2}(1+\xi)^{-k}
    }^2.
  \end{align*}
  It follows that, for any $ g \in \ell_\mathbb F^p(L^p(\Omega;L^q(\mathcal O;H))) $,
  \begin{align*}
    & \norm{\mathcal I(z)g}_{\ell^p(L^p(\Omega;L^q(\mathcal O)))}^p \\
    \leqslant{} &
    c I(\infty,\xi)^{p/2-1} \sum_{j=1}^\infty \sum_{k=0}^{j-1}
    \snmb{ \xi^{1/2} (1+\xi)^{k+1-j} }^2
    \norm{g_k}_{L^p(\Omega;L^q(\mathcal O;H))}^p \\
    ={} &
    c I(\infty,\xi)^{p/2-1} \sum_{k=0}^\infty \sum_{j=k+1}^\infty
    \snmb{ \xi^{1/2} (1+\xi)^{k+1-j} }^2
    \norm{g_k}_{L^p(\Omega;L^q(\mathcal O;H))}^p \\
    \leqslant{} &
    c I(\infty,\xi)^{p/2} \norm{g}_{\ell^p(L^p(\Omega;L^q(\mathcal O;H)))}^p \\
    ={} &
    c I(\infty,e^{-z}-1)^{p/2} \norm{g}_{\ell^p(L^p(\Omega;L^q(\mathcal O;H)))}^p,
  \end{align*}
  by the fact $ \xi = e^{-z} - 1 $. This leads to the bound
  \begin{align*}
    \norm{\mathcal I(z)}_{
      \mathcal L\big(
        \ell_\mathbb F^p(L^p(\Omega;L^q(\mathcal O;H))), \,
        \ell^p(L^p(\Omega;L^q(\mathcal O)))
      \big)
    } \leqslant c I(\infty,e^{-z}-1)^{1/2}.
  \end{align*}
  Therefore, the desired claim follows from the fact $ I(\infty,0) = 0 $ and the estimate
  \begin{align*} 
    \sup_{z \in \Upsilon\setminus\{0\}} \, I(\infty,e^{-z}-1)
    &= \sup_{z \in \Upsilon\setminus\{0\}} \,
    \sum_{k=0}^\infty \snmb{
      (e^{-z}-1)^{1/2} e^{kz}
    }^2 = \sup_{z \in \Upsilon\setminus\{0\}} \, \frac{\snm{e^{-z}-1}}{1 - \snm{e^{2z}}} < \infty,
  \end{align*}
  which is easily verified by the construction of $ \Upsilon $ as detailed in Subsection \ref{ssec:space-regu}.
\end{proof}

\begin{lemma}
  \label{lem:Fefferman-Stein}
  Let $ r \in (1,\infty) $ and $ s \in (1,\infty] $. For any $ g \in \ell^r(L^s(\mathcal O)) $, we have
  \[
    \sum_{j=1}^\infty  \nmB{
      \sup_{m \in \mathbb N} \frac1{1+m}
      \sum_{k=j}^{j+m} \snm{g_k}
    }_{L^s(\mathcal O)}^r \leqslant
    c \norm{g}_{\ell^r(L^s(\mathcal O))}^r.
  \]
\end{lemma}
\begin{proof}
  Define
  \[ 
    G(t) := \sup_{m \in \mathbb N}
    \frac1{(1+m)\tau} \int_t^{t+(1+m) \tau}
    \widetilde g(\beta) \, \mathrm{d}\beta.
  \]
  where $ \widetilde g(t) := \snm{g_j} $ for all $ t \in [{j\tau},{j\tau+\tau}) $
  and $ j \in \mathbb N $. For any $ j \geqslant 1 $, it is easily verified that
  \begin{align*}
    G({j\tau})(x) \leqslant \frac{c}\tau
    \int_{j\tau-\tau}^{{j\tau+\tau}} G(t)(x)
    \, \mathrm{d}t, \quad \forall x \in \mathcal O,
  \end{align*}
  which implies, through Minkowski's inequality and H\"older's inequality, that
  \begin{align*} 
    \norm{G({j\tau})}_{L^s(\mathcal O)}^r
    & \leqslant
    c \tau^{-r} \biggl[
      \int_{\mathcal O}
      \Big(
        \int_{j\tau-\tau}^{{j\tau+\tau}} G(t) \, \mathrm{d}t
      \Big)^s \, \mathrm{d}\mu
      \biggr]^{r/s} \\
    & \leqslant
    c \tau^{-r} \biggl[
      \int_{j\tau-\tau}^{{j\tau+\tau}} \norm{G(t)}_{L^s(\mathcal O)}
      \, \mathrm{d}t
    \biggr]^r \\
    & \leqslant
    c \tau^{-1} \int_{j\tau-\tau}^{{j\tau+\tau}}
    \norm{G(t)}_{L^s(\mathcal O)}^r
    \, \mathrm{d}t.
  \end{align*}
  Summing over $j$, we find
  \begin{align*} 
    \sum_{j = 1}^\infty \norm{G({j\tau})}_{L^s(\mathcal O)}^r
    &\leqslant c \tau^{-1} \int_{\mathbb R_{+}}
    \norm{G(t)}_{L^s(\mathcal O)}^r \, \mathrm{d}t \\
    &\leqslant c \tau^{-1} \norm{\widetilde g}_{L^r(\mathbb R_{+};L^s(\mathcal O))}^r
    \quad\text{(by \cite[Proposition~3.4]{Neerven2012})} \\
    & = c \norm{g}_{\ell^r(L^s(\mathcal O))}^r.
  \end{align*}
  The desired inequality then follows from  the fact
  \[
    G({j\tau}) = \sup_{m \in \mathbb N}
    \frac{1}{1+m} \sum_{k=j}^{j+m}
    \snm{g_k}, \quad \forall j \geqslant 1.
  \]
  This completes the proof.
\end{proof}

\begin{lemma}
  \label{lem:Pi_m}
  Assume that $ p,q \in [2,\infty) $. Define
  \[ 
    \{ \Pi_m \mid \, m \in \mathbb N \} \subset
    \mathcal L\big(
      \ell_{\mathbb F}^p(L^p(\Omega;L^q(\mathcal O;H))), \,
      \ell^p(L^p(\Omega;L^q(\mathcal O)))
    \big)
  \]
  as follows: for each $ m \in \mathbb N $ and
  $ g \in \ell_{\mathbb F}^p(L^p(\Omega;L^q(\mathcal O;H))) $,
  \begin{subequations}
    \label{eq:calII-def}
    \begin{numcases}{}
      (\Pi_mg)_0 := 0, \\
      (\Pi_mg)_j := \frac1{\sqrt{m+1}}
      \sum_{k=(j-1-m)\vee 0}^{j-1} g_k \frac{\delta W_k}{\sqrt\tau},
      \quad j \geqslant 1.
    \end{numcases}
  \end{subequations}
  Then $ \mathcal R(\{\Pi_m \mid m \in \mathbb N\}) $ is uniformly bounded
  with respect to the time step $ \tau $.
\end{lemma}

\begin{proof}
  Following the proof of \cite[Theorem~3.1]{Neerven2012}, we only present
  a brief derivation.
  Consider an arbitrary positive integer $ N $. Let $ (r_n)_{n=1}^N $ denote a sequence of
  independent, symmetric $\{-1,1\}$-valued random variables defined on the probability
  space $ \Omega_r $. We denote the expectation with respect
  to this probability space by $ \mathbb E_r $. Given a sequence $ (g^n)_{n=1}^N $ in
  $ \ell_\mathbb F^p(L^p(\Omega;L^q(\mathcal O;H))) $ and a sequence $ (m_n)_{n=1}^N $ in
  $ \mathbb N $, we proceed as follows:
  \begin{align*}
    & \biggl[
      \mathbb{E}_r \Bigl\| \sum_{n=1}^N r_n \Pi_{m_n} g^n \Bigr\|_{\ell^p(L^p(\Omega; L^q(\mathcal{O})))}^2
    \biggr]^{1/2} \\
    \leqslant{}
    & c \biggl[ \mathbb{E}_r \Bigl\| \sum_{n=1}^N r_n \Pi_{m_n} g^n \Bigr\|_{\ell^p(L^p(\Omega; L^q(\mathcal{O})))}^p \biggr]^{1/p} \\
    = {} & c \biggl[
      \sum_{j=1}^{\infty} \mathbb{E} \mathbb{E}_r \Bigl\|
      \sum_{n=1}^N \frac{r_n}{\sqrt{1+m_n}} \sum_{k=j-1-m_n\vee 0}^{j-1} g_k^n \delta W_k / \sqrt{\tau} \Bigr\|_{L^q(\mathcal{O})}^p 
    \biggr]^{1/p} \\
    \leqslant{}
         & c \biggl[
           \sum_{j=1}^{\infty} \mathbb{E}_r \mathbb{E} \Big[ \int_{\mathcal{O}} \Big( \sum_{k=0}^{j-1} \Bigl\| \sum_{n=1}^N \frac{r_n}{\sqrt{1+m_n}} \mathbbm{1}_{j-1-m_n \leqslant k} g_k^n \Bigr\|_H^2 \Big)^{q/2} \, \mathrm{d}\mu \Big]^{p/q} 
         \biggr]^{1/p},
    \end{align*}
    where the first inequality utilizes the Kahane-Khintchine inequality (see, e.g., \cite[Theorem~6.2.4]{HytonenWeis2017}),
    and the final inequality employs Lemma \ref{lem:integral}.
    Here, \( \mathbbm{1}_{j-1-m_n \leqslant k} \) denotes the indicator 
    function, which equals \(1\) if \( j-1-m_n \leqslant k \), and \(0\) otherwise.
    Applying the Kahane-Khintchine inequality again, we obtain, for any \( j \in \mathbb{N}_{>0} \),
    \begin{align*}
        & \mathbb{E}_r \biggl[
          \int_{\mathcal{O}} \Big(
            \sum_{k=0}^{j-1} \Bigl\| \sum_{n=1}^N \frac{r_n}{\sqrt{1+m_n}} \mathbbm{1}_{j-1-m_n \leqslant k} g_k^n \Bigr\|_H^2
          \Big)^{q/2} \, \mathrm{d}\mu
        \biggr]^{p/q} \\
      \leqslant{} & c \Bigl\| \sum_{n=1}^N \frac{1}{1+m_n} \sum_{k=j-1-m_n\vee 0}^{j-1} \|g_k^n\|_H^2 \Bigr\|_{L^{q/2}(\mathcal{O})}^{p/2}.
    \end{align*}
    Combining the above estimates, we arrive at
    \begin{equation}
      \label{eq:426}
      \begin{aligned}
        & \biggl[
          \mathbb E_r \nmB{\sum_{n=1}^N r_n \Pi_{m_n}g^n}_{
            \ell^p(L^p(\Omega;L^q(\mathcal O)))
          }^2
        \biggr]^{1/2} \\
        \leqslant{} &
        c \biggl[
          \sum_{j=1}^\infty \nmB{
            \sum_{n=1}^N \frac1{1+m_n} \sum_{k=j-1-m_n\vee 0}^{j-1}
            \norm{g_k^n}_H^2
          }_{L^{p/2}(\Omega;L^{q/2}(\mathcal O))}^{p/2}
        \biggr]^{1/p}.
      \end{aligned}
    \end{equation}
    For any $ Z \in \ell^{(p/2)'}(L^{(p/2)'}(\Omega;L^{(q/2)'}(\mathcal O))) $,
    a direct calculation gives
    \begin{align*}
      \sum_{j=1}^\infty \dualB{
        \sum_{n=1}^N \frac1{1+m_n} \sum_{k=j-1-m_n\vee 0}^{j-1}
        \norm{g_k^n}_H^2, \, Z_j
      } &=
      \sum_{k=0}^\infty \sum_{n=1}^N \sum_{j=k+1}^{k+1+m_n}
      \frac1{1+m_n} \dualB{
        \norm{g_k^n}_H^2, \, Z_j
      } \\
      &=
      \sum_{k=0}^\infty \sum_{n=1}^N \dualB{
        \norm{g_k^n}_H^2, \, \frac1{1+m_n}
        \sum_{j=k+1}^{k+1+m_n} Z_j
      } \\
      &\leqslant
      \sum_{k=0}^\infty \dualB{
        \sum_{n=1}^N \norm{g_k^n}_H^2,
        \, \sup_{m\in\mathbb N}\frac1{1+m}
        \sum_{j=k+1}^{k+1+m} \snm{Z_j}
      },
    \end{align*}
    where $ \dual{\cdot,\cdot} $ denotes the duality pairing between $ L^{p/2}(\Omega;L^{q/2}(\mathcal O)) $
    and $ L^{(p/2)'}(\Omega;L^{(q/2)'}(\mathcal O)) $.
    By H\"older's inequality and \cref{lem:Fefferman-Stein}, it follows that
    \begin{align*}
      & \sum_{j=1}^\infty \dualB{
        \sum_{n=1}^N \frac1{1+m_n}
        \sum_{k=j-1-m_n\vee 0}^{j-1}
        \norm{g_k^n}_H^2, \, Z_j
      } \\
      \leqslant{} &
      c\biggl[
        \sum_{k=0}^\infty \nmB{
          \sum_{n=1}^N \norm{g_k^n}_H^2
        }_{L^{p/2}(\Omega;L^{q/2}(\mathcal O))}^{p/2}
      \biggr]^{2/p} \times \norm{Z}_{\ell^{(p/2)'}(L^{(p/2)'}(\Omega;L^{(q/2)'}(\mathcal O)))}.
    \end{align*}
    Invoking the duality principle then yields
    \begin{align*}
    & \biggl[
      \sum_{j=1}^\infty \nmB{
        \sum_{n=1}^N \frac1{1+m_n} \sum_{k=j-1-m_n\vee 0}^{j-1}
        \norm{g_k^n}_H^2
      }_{L^{p/2}(\Omega;L^{q/2}(\mathcal O))}^{p/2}
    \biggr]^{2/p} \\
      \leqslant{} &
      c\biggl[
        \sum_{k=0}^\infty \nmB{
          \sum_{n=1}^N \norm{g_k^n}_H^2
        }_{L^{p/2}(\Omega;L^{q/2}(\mathcal O))}^{p/2}
      \biggr]^{2/p}.
    \end{align*}
    In conjunction with \cref{eq:426}, this leads to
    \begin{align*}
      \biggl[
        \mathbb E_r \nmB{\sum_{n=1}^N r_n \Pi_{m_n}g^n}_{
          \ell^p(L^p(\Omega;L^q(\mathcal O)))
        }^2
      \biggr]^{1/2} \leqslant
      c\biggl[
        \sum_{k=0}^\infty \nmB{
          \sum_{n=1}^N \norm{g_k^n}_H^2
        }_{L^{p/2}(\Omega;L^{q/2}(\mathcal O))}^{p/2}
      \biggr]^{1/p}.
    \end{align*}
    On the other hand, using the Kahane-Khintchine inequality gives
    \[
      \biggl[ \mathbb{E}_r \Bigl\| \sum_{n=1}^N r_n g^n \Bigr\|_{\ell^p(L^p(\Omega; L^q(\mathcal{O}; H)))}^2 \biggr]^{1/2}
      \geqslant c \biggl[
        \sum_{j=0}^{\infty} \Bigl\| \sum_{n=1}^N \|g_j^n\|_H^2 \Bigr\|_{L^{p/2}(\Omega; L^{q/2}(\mathcal{O}))}^{p/2}
      \biggr]^{1/p}.
    \]
    Consequently,
    \begin{align*}
      \biggl[
        \mathbb E_r \nmB{\sum_{n=1}^N r_n \Pi_{m_n}g^n}_{
          \ell^p(L^p(\Omega;L^q(\mathcal O)))
        }^2
      \biggr]^{1/2} \leqslant
      c \biggl[
        \mathbb E_r \nmB{\sum_{n=1}^N r_n g^n}_{
          \ell^p(L^p(\Omega;L^q(\mathcal O;H)))
        }^2
      \biggr]^{1/2}.
    \end{align*}
    Since the above generic positive constant \( c \) is independent of \( \tau \),
    and given that \( N \) is an arbitrary positive integer
    while \( (g^n)_{n=1}^N \) is an arbitrary sequence in
    \( \ell_{\mathbb F}^p(L^p(\Omega; L^q(\mathcal O; H))) \),
    $ (m_n)_{n=1}^N $ is an arbitrary sequence in $ \mathbb N $,
    and $ (r_n)_{n=1}^N $ is an arbitrary sequence of independent symmetric 
    $ \{-1,1\}$-valued random variables,
    it follows that \( \mathcal R(\{\Pi_m \mid m \in \mathbb N\}) \)
    is uniformly bounded with respect to \( \tau \).
    This completes the proof.
\end{proof}


\begin{thebibliography}{10}

\bibitem{Bessaih2019}
H.~Bessaih.
\newblock {Strong $ L^2 $ convergence of time numerical schemes for the
  stochastic two-dimensional Navier-Stokes equations}.
\newblock {\em IMA J. Numer. Anal.}, 39:2135--2167, 2019.

\bibitem{Blomker2013}
D.~Bl\"omker and A. Jentzen.
\newblock {Galerkin approximations for the stochastic Burgers equation}.
\newblock {\em SIAM J. Numer. Anal.}, 51:694--715, 2013.

\bibitem{Blunck2001}
S.~Blunck.
\newblock Maximal regularity of discrete and continuous time evolution
  equations.
\newblock {\em Studia Math.}, 146:157--176, 2001.

\bibitem{Breit2021}
D.~Breit and A.~Dodgson.
\newblock Convergence rates for the numerical approximation of the 2d
  stochastic {N}avier-{S}tokes equations.
\newblock {\em Numer. Math.}, 147:553--578, 2021.

\bibitem{Breit2019}
Z.~Brz\'ezniak, J.~Cui, and J.~Hong.
\newblock Strong convergence rates of semidiscrete splitting approximations for the
stochastic Allen-Cahn equation.
\newblock {\em IMA J. Numer. Anal.}, 39:2096--2134, 2019.


\bibitem{CarelliProhl2012}
Z.~Brz\'ezniak, E.~Carelli, and A.~Prohl.
\newblock Time-splitting methods to solve the stochastic incompressible
  {S}tokes equation.
\newblock {\em SIAM J. Numer. Anal.}, 50:2917--2939, 2012.



\bibitem{Prohl2013}
Z.~Brz\'ezniak, E.~Carelli, and A.~Prohl.
\newblock Finite-element-based discretizations of the incompressible
  {N}avier–{S}tokes equations with multiplicative random forcing.
\newblock {\em IMA J. Numer. Anal.}, 33:771--824, 2013.

\bibitem{Prohl2012}
E.~Carelli and A.~Prohl.
\newblock Rates of convergence for discretizations of the stochastic
  incompressible {N}avier-{S}tokes equations.
\newblock {\em SIAM J. Numer. Anal.}, 50:2467--2496, 2012.


\bibitem{Cox2010}
S.~Cox and J. van Neerven.
\newblock Convergence rates of the splitting scheme for parabolic linear stochastic {C}auchy problems.
\newblock {\em SIAM J. Numer. Anal.}, 48:428--451, 2010.

\bibitem{Cox2013}
S.~Cox and J. van Neerven.
\newblock Pathwise H\"older convergence of the implicit-linear {E}uler scheme for semi-linear SPDEs with multiplicative noise.
\newblock {\em Numer. Math.}, 125:259--345, 2013.



\bibitem{Denk2013book}
R.~Denk and M.~Kaip.
\newblock {\em {General Parabolic Mixed Order Systems in $ L_p $ and
  Applications, Operator Theory: Advances and Applications, vol. 239}}.
\newblock Springer, Cham, 2013.


\bibitem{Diening2023}
L.~Diening, M.~Hofmanov\'a, and J.~Wichmann.
\newblock An averaged space-time discretization of the stochastic $p$-Laplace system.
\newblock {\em Numer. Math.}, 153:557--609.



\bibitem{Gyongy1999}
I.~Gy\"ongy.
\newblock Lattice approximations for stochastic quasi-linear parabolic partial differential equations driven by space-time white noise.
\newblock {\em Potential Anal.}, 11:1--37, 1999.

\bibitem{Gyongy2003}
I.~Gy\"ongy and N.~Krylov.
\newblock On the splitting-up method and stochastic partial differential equations.
\newblock {\em Ann. Probab.}, 31:564--591, 2003.


\bibitem{Gyongy2009}
I.~Gy\"ongy and I.~Millet.
\newblock Rate of convergence of space time approximations for stochastic evolution equations.
\newblock {\em Potential Anal.}, 30:29--64, 2009.


\bibitem{Gyongy1997}
I.~Gy\"ongy and D.~Nualart.
\newblock Implicit scheme for stochastic parabolic partial differential equations driven by space-time white noise.
\newblock {\em Potential Anal.}, 7:725--757, 1997.


\bibitem{Haase2006}
M.~Haase. 
\newblock {\em The functional calculus for sectorial operators}.
\newblock Birkhäuser Basel, 2006.


\bibitem{HytonenWeis2016}
T.~Hyt\"onen, J.~van Neerven, M.~Veraar, and L.~Weis.
\newblock {\em Analysis in Banach spaces, Volume I: Martingales and Littlewood-Paley Theory}.
\newblock Springer, Cham, 2016.

\bibitem{HytonenWeis2017}
T.~Hyt\"onen, J.~van Neerven, M.~Veraar, and L.~Weis.
\newblock {\em Analysis in Banach spaces, Volume II: Probabilistic Methods and Operator Theory}.
\newblock Springer, Cham, 2017.

\bibitem{HytonenWeis2023}
T.~Hyt\"onen, J.~van Neerven, M.~Veraar, and L.~Weis.
\newblock {\em Analysis in Banach spaces, Volume III: Harmonic Analysis and Spectral Theory}.
\newblock Springer, Cham, 2023.

\bibitem{Klioba2024}
K.~Klioba and M. Veraar.
\newblock Temporal approximation of stochastic evolution equations with irregular nonlinearities.
\newblock {\em J. Evol. Equ.}, 24:43, 2024.


\bibitem{Kemmochi2016}
T.~Kemmochi.
\newblock Discrete maximal regularity for abstract cauchy problems.
\newblock {\em Studia Math.}, 234:241--263, 2016.

\bibitem{Kemmochi2018}
T.~Kemmochi and N.~Saito.
\newblock Discrete maximal regularity and the finite element method for
  parabolic equations.
\newblock {\em Numer. Math.}, 138:905--937, 2018.

\bibitem{Lubich2016}
B.~Kov\'acs, B.~Li, and C.~Lubich.
\newblock A-stable time discretizations preserve maximal parabolic regularity.
\newblock {\em SIAM J. Numer. Anal.}, 54:3600--3624, 2016.

\bibitem{Kruse2014book}
R.~Kruse.
\newblock {\em Strong and weak approximation of semilinear stochastic evolution
  equations}.
\newblock Springer, Cham, 2014.

\bibitem{Kunstmann2004}
P.C. Kunstmann and L.~Weis.
\newblock {\em {Maximal $ L_p $-regularity for Parabolic Equations, Fourier
  Multiplier Theorems and $H^\infty$-functional Calculus}}, pages 65--311.
\newblock Springer, Berlin, 2004.

\bibitem{Kalton2006}
N. Kalton, P. Kunstmann, and L. Weis.
\newblock Perturbation and interpolation theorems for the $H^\infty$-calculuus with applications to differential operators.
\newblock {\em Math. Ann.}, 336:747--801, 2006.


\bibitem{Kazashi2018}
Y.~Kazashi.
\newblock Discrete maximal regularity of an implicit Euler-Maruyama scheme with non-uniform time discretization for a class of stochastic partial differential equations.
\newblock {\em Commun. Probab.}, 29:1--14, 2018.


\bibitem{Vexler_Lp_2017}
D.~Leykekhman and B.~Vexler.
\newblock Discrete maximal parabolic regularity for Galerkin finite element
  methods.
\newblock {\em Numer. Math.}, 135:923--952, 2017.

\bibitem{LiB2017Math}
B.~Li and W.~Sun.
\newblock {Maximal $ L^p $ analysis of finite element solutions for parabolic
  equations with nonsmooth coefficients in convex polyhedra}.
\newblock {\em Math. Comp.}, 86:1071--1102, 2017.

\bibitem{LiB2017SIAM}
B.~Li and W.~Sun.
\newblock Maximal regularity of fully discrete finite element solutions of
  parabolic equations.
\newblock {\em SIAM J. Numer. Anal.}, 55:521--542, 2017.

\bibitem{LiZhou2024-time}
B.~Li and Q.~Zhou
\newblock {\em Pathwise uniform convergence of a full discretization for a three-dimensional stochastic Allen-Cahn equation with multiplicative noise}.
\newblock arXiv:2405.03016, 2024.

\bibitem{Pruss2016}
J.~Pr\"uss and G.~Simonett.
\newblock {\em Moving interfaces and quasilinear parabolic evolution
  equations}.
\newblock Birkh\"auser Basel, 2016.

\bibitem{Sinha2017book}
K.~B. Sinha and S.~Srivastava.
\newblock {\em Theory of semigroups and applications}.
\newblock Springer, Singapore, 2017.


\bibitem{Neerven2022}
J.~van Neerven and M.~Veraar.
\newblock {Maximal inequalities for stochastic convolutions and pathwise uniform convergence of time discretization schemes}.
\newblock {\em Stoch PDE: Anal. Comp.}, 10:516--581, 2022.

\bibitem{Neerven2007}
J.~van Neerven, M.~Veraar, and L. Weis.
\newblock{Stochastic integration in UMD Banach spaces.}
\newblock{\em Ann. Probab.}, 35:1438–1478, 2007.

\bibitem{Neerven2012b}
J.~van Neerven, M.~Veraar, and L.~Weis.
\newblock {Maximal $ L^p $-regularity for stochastic evolution equations}.
\newblock {\em SIAM J. Math. Anal.}, 44:1372--1414, 2012.

\bibitem{Neerven2012}
J.~van Neerven, M.~Veraar, and L.~Weis.
\newblock {Stochastic maximal $ L^p $-regularity}.
\newblock {\em Ann. Probab.}, 40:788--812, 2012.



\bibitem{Prato2014}
  G.~Da Prato and J.~Zabczyk.
  \newblock{\em Stochastic equations in infinite dimensions},
  \newblock Cambridge University Press, 2014.



\bibitem{Weis2001}
L.~Weis.
\newblock {Operator-valued Fourier multiplier theorems and maximal $ L_p
  $-regularity}.
\newblock {\em Math. Ann.}, 319:735--758, 2001.

\bibitem{Yan2005}
Y.~Yan.
\newblock Galerkin finite element methods for stochastic parabolic partial
  differential equations.
\newblock {\em SIAM J. Numer. Anal.}, 43:1363--1384, 2005.

\bibitem{Zhang2017book}
Z.~Zhang and G.~E. Karniadakis.
\newblock {\em Numerical methods for stochastic partial differential equations
  with white noise}.
\newblock Springer, Cham, 2017.

\end{thebibliography}
\end{document}